\documentclass[11pt]{amsart}

\usepackage{amsmath,amsthm,amssymb,amssymb, paralist, xspace, graphicx, url, amscd, euscript, mathrsfs,stmaryrd,epic,eepic,color}
\usepackage{verbatim}

\usepackage{tikz}

 \usepackage{pgfplots}

 \tikzset{dot/.style={circle,fill=#1,inner sep=0,minimum size=4pt}}

\usepackage[all]{xy}
\usepackage{amsmath,amscd}
\SelectTips{cm}{}

\usepackage[colorlinks=true]{hyperref}

\newcommand{\Qile}[1]{}
\newcommand{\Yi}[1]{}

\usepackage{pstricks}

\numberwithin{equation}{subsection}

1

\numberwithin{subsection}{section}

\allowdisplaybreaks[1]

\newenvironment{enumeratea}
{\begin{enumerate}[\upshape (a)]}
{\end{enumerate}}

\newenvironment{enumerate1}
{\begin{enumerate}[\upshape (1)]}
{\end{enumerate}}

\newtheorem*{namedtheorem}{\theoremname}
\newcommand{\theoremname}{testing}

\theoremstyle{plain}

\newtheorem{thm}{Theorem}[section]
\newtheorem{proposition}[thm]{Proposition}
\newtheorem{proposition-definition}[thm]{Proposition-Definition}
\newtheorem{lemma-definition}[thm]{Lemma-Definition}
\newtheorem{corollary}[thm]{Corollary}
\newtheorem{lemma}[thm]{Lemma}

\theoremstyle{definition}
\newtheorem{definition}[thm]{Definition}
\newtheorem{term}[thm]{Terminology}
\newtheorem{notation}[thm]{Notation}
\newtheorem{example}[thm]{Example}

\newtheorem{remark}[thm]{Remark}

\newtheorem{hyp}[thm]{Hypothesis}

\theoremstyle{remark}

\numberwithin{thm}{section}

\newcommand\ocM{\overline{\mathcal{M}}}

\newcommand\ocN{\overline{\mathcal{N}}}

\newcommand\cA{\mathcal{A}}

\newcommand\cC{\mathcal{C}}

\newcommand\cF{\mathcal{F}}

\newcommand\cM{\mathcal{M}}

\newcommand\cO{\mathcal{O}}

\newcommand\cT{\mathcal{T}}

\newcommand\cV{\mathcal{V}}

\newcommand\N{\mathcal{N}}

\def\O{\mathcal{O}}

\def\P{\mathbb{P}}
\def\A{\mathbb{A}}
\def\M{\mathcal{M}}
\def\N{\mathbb{N}}

\def\Q{\mathbb{Q}}
\def\X{\mathbf{X}}
\def\o{\overline}
\def\u{\underline}

\newcommand\uB{\underline{B}}
\newcommand\uC{\underline{C}}
\newcommand\uD{\underline{D}}

\newcommand\uF{\underline{F}}
\newcommand\uf{\underline{f}}

\newcommand\uM{\underline{M}}

\newcommand\uP{\underline{P}}

\newcommand\uR{\underline{R}}
\newcommand\uS{\underline{S}}
\newcommand\uT{\underline{T}}
\newcommand\uU{\underline{U}}
\newcommand\uV{\underline{V}}

\newcommand\uX{\underline{X}}
\newcommand\uY{\underline{Y}}
\newcommand\uZ{\underline{Z}}

\renewcommand\AA{\mathbb{A}}

\newcommand\NN{\mathbb{N}}
\newcommand\PP{\mathbb{P}}
\newcommand\QQ{\mathbb{Q}}

\newcommand\ZZ{\mathbb{Z}}

\newcommand\fC{\mathfrak{C}}

\newcommand\fM{\mathfrak{M}}

\newcommand\fT{\mathfrak{T}}

\newcommand\fX{\mathfrak{X}}

\newcommand\kk{\mathbf{k}}  

\newcommand\arr{\ifinner\to\else\longrightarrow\fi}

\def\displaytimes_#1{\mathrel{\mathop{\times}\limits_{#1}}}

\def\displayotimes_#1{\mathrel{\mathop{\bigotimes}\limits_{#1}}}

\newcommand\CH{\operatorname{CH}}

\newcommand\pic{\operatorname{Pic}}

\newcommand\spec{\operatorname{Spec}}

\newcommand\ch{\operatorname{char}}

\newcommand\indlim{\varinjlim}

\newdir{ >}{{}*!/-5pt/@{>}}

\newcommand\doublelong[2]{\mathbin{\xymatrix{{}\ar@<3pt>[r]^{#1}
\ar@<-3pt>[r]_{#2}&}}}

\newlength{\ignora}

\newcommand{\diff}{\operatorname{d}}

\newcommand{\R}{}

\renewcommand{\setminus}{\smallsetminus}

\newcommand{\Z}{\mathbb{Z}}

\numberwithin{equation}{subsection}

\newcommand{\NE}{\operatorname{NE}}

\newcommand{\co}{\mathfrak{c}} 

\begin{document}

\title{$\AA^1$-curves on log smooth varieties}

\author{Qile Chen}

\author{Yi Zhu}

\thanks{Chen is partially supported by the Simons foundation and NSF grant DMS-1403271.}

\address[Chen]{Department of Mathematics\\
Boston College\\
Chestnut Hill, Ma 02467\\
U.S.A.}
\email{qile.chen@bc.edu}

\address[Zhu]{Department of Mathematics\\
University of Utah\\
Room 233\\
155 S 1400 E \\
Salt Lake City, UT 84112\\
U.S.A.}
\email{yzhu@math.utah.edu}

\date{\today}

\subjclass[2010]{Primary 14M22, Secondary 14M17}

\begin{abstract}
In this paper, we study $\mathbb{A}^1$-connected varieties from log geometry point of view, and prove a criterion for $\mathbb{A}^1$-connectedness. As applications, we provide many interesting examples of $\A^1$-connected varieties in the case of complements of ample divisors, and the case of homogeneous spaces. We also obtain a logarithmic version of Hartshorne's conjecture characterizing projective spaces and affine spaces.


\end{abstract}
\maketitle

\tableofcontents

\section{Introduction}\label{sec:intro}

Throughout this paper, we work over an algebraically closed field $\kk$ of characteristic $\ch \kk \geq 0$. 

\subsection{$\A^1$-curves}


Rational curves on projective varieties have been intensively studied in algebraic geometry. For varieties admitting lots of rational curves, rational connectedness plays a central \R{role}  in the study of birational geometry of higher dimensional varieties \cite{Mori79, Campana, KMM}. Varieties that are (separably) rationally connected admit nice arithmetic properties \cite{Kollar-local, Kollar-Szabo, GHS, dJS1, weak-app}.



One would like to study the analogue of rational curves and rational connectedness on open varieties. Keel and McKernan \cite{KM} suggest that the right notion of ``rational curve'' on a non-proper variety should include both rational curves and {\em $\AA^1$-curves}. Campana \cite{campana04,Campana11} further introduced the notion of {orbifold rational curves} which include $\A^1$-curves as an important case.




\begin{definition}\label{def:A1-curve-scheme}
An \emph{$\AA^1$-curve} on a scheme $\uU$ is a non-constant, proper morphism $\u{f}: \A^1 \to \uU$. A {\em family of $\A^1$-curves} over a scheme $\uT$ is a morphism $\uf: \A^1 \times \uT \to \uU$ such that for each geometry point $t \in \uT$, $\uf_t$ is an $\A^1$-curve.

\end{definition}


$\A^1$-curves behave in many ways similar to the case of rational curves. They play an essential role in the classification theory of open algebraic surfaces \cite{Kawamata79,Miyanishi-T2,Miyanishi-T1,KM,Zhu14}. In \cite{CZ}, $\AA^1$-curves are used to \R{produce} rational curves on projective varieties via degenerations. 


Despite their importance, $\A^1$-curves are much more difficult to construct than rational curves.  Some methods of producing $\A^1$-curves have been studied in \cite{KM,  Lu-Zhang, MG, Svaldi, CZ}. However, the log Bend-and-Break conjecture \cite[1.11]{KM} remains largely open.







\subsection{Main result and new examples of $\A^1$-connected varieties}
%
%

Similar to uniruledness and rational connectedness \cite[IV.(1.1.1), (3.2.2), (3.2.3)]{kollar}, we define:



\begin{definition}\label{def:intro-uniruled-connected}
Let $\uU$ be an algebraic variety. 
\begin{enumerate1}
\item$\uU$ is \emph{$\A^1$-uniruled} (resp. \emph{separably $\A^1$-uniruled}) if there exists a variety $\uT$ and a $\uT$-family of \emph{$\AA^1$-curves} $
\u{f}: \uT\times \A^1 \to \uU
$ with $\uf$  dominant (resp. separable dominant).

\item $\uU$ is \emph{$\A^1$-connected} (resp. \emph{separably $\A^1$-connected}) if there exists a variety $\uT$ and
a $\uT$-family of \emph{$\AA^1$-curves} $\u{f}: \uT\times \A^1 \to \uU$ such that 
\[
	(\u{f})^{(2)} := \u{f}\times_{\uT} \u{f}: \uT\times \A^1 \times \A^1 \to \uU\times \uU	
\]
	is dominant (resp. dominant and separable).
\end{enumerate1}
\end{definition}

$\A^1$-connected varieties are analogue of rationally connected varieties in the non-proper setting. They share many similar geometric properties. Furthermore, $\A^1$-connectedness has applications in studying integral points over function fields of complex curves \cite{rankone, strong}. 

In this paper, we provide a new criterion of $\A^1$-connectedness:


%


\begin{thm}\label{thm:comb-smoothing}
\R{Let $X$ be a log smooth log variety with a proper, separably rationally connected, and fully free center $Y$.} Then 
the loci $X^{\circ} \subset X$ with the trivial log structure is separably $\A^1$-connected.
\end{thm}


\begin{remark}
We explain the terminologies in the simple normal crossings case. Let $\uX$ be a smooth compacitification of $X^{\circ}$ with simple normal crossings boundary $\uD = \cup_i \uD_i$, and $X$ be the log smooth variety associated to the pair $(\uX, \uD)$, see Example \ref{ex:toroidal}. A {\em center} $Y$ is a deepest stratum of the boundary divisor $\uD$. It is {\em fully free} if there exist free rational curves $R_1, \cdots, R_m$ on $Y$ such that 
\begin{enumerate}
 \item $a_{ij} := \uD_i \cdot R_j  \geq 0$ for each $i,j$,
 \item The matrix of integers $(a_{ij})$ over $\kk$ is of rank $r$ where $r$ is the number of $\uD_i$'s containing $Y$.
\end{enumerate}
We refer to Definition \ref{def:full-free} for fully free centers in the general setting. 
\end{remark}


\R{
Theorem \ref{thm:comb-smoothing} is proved in Section \ref{ss:comb-proof}. It provides many interesting examples of $\A^1$-connected varieties. Combining Theorem \ref{thm:comb-smoothing} and Proposition \ref{prop:new-fully}, we prove $\A^1$-connectedness for complements of ample divisors.

\begin{thm}\label{thm:new}
Suppose $\ch\kk =0$. Let $\uD=\uD_1+\cdots+\uD_r$ be a simple normal crossings divisor on a smooth projective variety $\uX$. Assume that	
\begin{enumerate}
		\item $\uD_i$ is smooth and ample for every $i$;
		\item $\uD_1,\cdots, \uD_r$ are linearly independent in $N^{1}_\Q(\uX)$;
		\item the center $\cap_{i=1}^r \uD_i$ is rationally connected of dimension $\ge 2$.
	\end{enumerate}
	Then $\uX\setminus \uD$ is $\A^1$-connected.
\end{thm}

\begin{remark}
By \cite[Lemma 2]{Iitaka-K3} and Corollary \ref{cor:van}, condition (2) in Theorem \ref{thm:new} is a necessary condition for $\A^1$-connectedness. 

\end{remark}


Projective homogeneous spaces \R{equipped} with an action of \R{a} reductive group $G$ are rationally connected. Their generalization in the quasi-projective setting are {sober spherical homogeneous spaces} of the form $G/H$, studied by Luna-Vust \cite{Luna-Vust}, Knop \cite{Knop} and Brion \cite{Brion-spherical, Brion-log-homogeneous}. They admit canonical toroidal compactifications, called {\em wonderful compactifications}.


\begin{thm}[See Theorem \ref{thm:char0}]\label{thm:intro-spherical}
Suppose $\ch\kk=0$. Let $G/H$ be a non-proper sober spherical homogeneous space under a reductive group $G$. Suppose all colors of $G/H$ are of type (b). Then $G/H$ are $\A^1$-connected.

\end{thm}

Type (b) condition, introduced by Luna \cite[Section 1.4]{Luna} (see also Definition \ref{def:b}), is satisfied by a large class of examples, including all semi-simple algebraic groups. This condition implies that the center of the wonderful compactification is fully free. In arbitrary characteristic, we obtain similar result for semisimple groups.
\begin{thm}[See Theorem \ref{thm:G}]\label{thm:intro-group}
A semisimple algebraic group $G$ is separably $\A^1$-connected if $\ch\kk\nmid |\pi_1(G)|$. In particular, simply connected semisimple algebraic groups are separably $\A^1$-connected.
\end{thm}
}

%

Indeed, Theorem \ref{thm:char0} and Theorem \ref{thm:G} give a complete classification of $\A^1$-curve classes on such homogeneous spaces, which further implies that \R{``interior effective'' curve classes are represented by log rational curves}:
\R{
\begin{thm}
Let $X^{\circ}$ be either $G/H$ as in Theorem \ref{thm:intro-spherical} in $\ch \kk = 0$, or a semisimple algebraic group in $\ch \kk \geq 0$. Denote by $\uX$ the wonderful compactification of $X^{\circ}$. Let $\uF \subset \uX$ be the closure of a possibly open curve on $X^{\circ}$, which intersects the boundary non-trivially (resp. trivially). Then the curve class $[F]$ in $N_1(\uX)$ is represented by the closure of an $\A^1$-curve (resp. $\PP^{1}$-curve) on $X^{\circ}.$
\end{thm}
}


\subsection{Log Hartshorne conjecture}
Combining Mori's idea with the theory of $\AA^1$-curves and Keel-McKernan's work \cite{KM}, we identify projective spaces as a distinguished compactification of affine spaces as follows:

\begin{thm}[See Theorem \ref{thm:projective-space}]\label{thm:logH}
	Let $X=(\uX,\uD)$ be a log smooth projective log variety. If the log tangent bundle is ample, then $(\uX,\uD)$ is isomorphic to either the pair $(\P^n,\emptyset)$ or $(\P^n,\text{a hyperplane})$. 
\end{thm}

The proof of Theorem \ref{thm:projective-space} will be given in Section \ref{sec:positive-tangent}. When the underlying space $\uX$ is smooth, a simple proof of Theorem \ref{thm:logH} using the result \cite[Theorem 1.1]{CP} can be given. 

\subsection{The logarithmic method}\label{ss:open-log}
The key to the proof of Theorem \ref{thm:comb-smoothing} is to embrace the log geometry. Our method, which is a continuation of our previous work \cite{CZ}, may be as interesting as the theorem itself.

Let $X^{\circ}$ be a smooth variety. Suppose $X$ is a proper, log smooth variety with the loci of trivial log structures given by $X^{\circ} \subset X$. Denote by $\PP^1_{\infty}$ the log scheme associated to the pair $(\PP^1, \infty)$ where $\infty \in \PP^1$ is a marking. Note that an $\A^1$-curve on $X^{\circ}$ determines a unique $\A^1$-curve on $X$ as below:

\begin{definition}\label{def:A1-curve}
A {\em log rational curve} on $X$ is a non-constant morphism of log schemes $f: \P^1_{\infty} \to X$. It is called an {\em $\A^1$-curve} if $f(\infty) \notin X^{\circ}$. Otherwise, it is called a {\em $\PP^1$-curve}.
\end{definition}

\begin{definition}\label{def:free-A1-curve}
A log rational curve $f: \P^1_\infty\to X$ is {\em free (resp. very free)} if $f^*T_{X}$ is semi-positive (resp. ample), where $T_X$ is the log tangent bundle of $X$.
\end{definition}

Using log deformation theory, we observe that
\begin{proposition}[See Proposition \ref{prop:open-A1-uniruled}, \ref{prop:open-A1-connected}, and \ref{prop:A1-uniruled}] {} \ 
\begin{enumerate1}
 \item $X^{\circ}$ is separably $\A^1$-uniruled if and only if $X$ has a free $\A^1$-curve.
 \item $X^{\circ}$ is separably $\A^1$-connected if and only if $X$ has a very free $\A^1$-curve.
\end{enumerate1}
\end{proposition}

To prove Theorem \ref{thm:comb-smoothing}, we use stable log maps of \cite{GS, Chen, AC}. We construct a degenerate stable log map with the underlying stable map lying in the center. By analyzing the log deformation theory, we show that such stable log map can be deformed to a very free $\A^1$-curve, see Section \ref{sec:DF-var}, \ref{sec:comb}.

\subsection{Notations}


Throughout this paper, all log structures are assumed to be fine and saturated \cite[2]{KKato}. Some useful notions of logarithmic geometry will be reviewed in Section \ref{sec:definition}. We refer to \cite{KKato} for the basics of logarithmic geometry. 

Capital letters such as $X, Y$ are reserved for log schemes with underlying \R{schemes} denoted by $\uX$ and $\uY$ respectively.



%
%
%

\subsection*{Acknowledgments}
The authors are grateful to Professor Dan Abramovich, Steffen Marcus, and Jonathan Wise for useful discussions on log \'etale resolution. In the collaboration with them on \cite{log-bound}, we learned the idea of log \'etale descent, which greatly inspires our construction in the current paper. \R{During} the preparation of this paper, they received a lot of help from Johan de Jong, Yi Hu, Mathieu Huruguen, J\'anos Koll\'ar, Jason Starr, Michael Thaddeus, and Xinwen Zhu.  The authors would like to express their thanks to them. A large part of our work has been done during the first author's visit of the Math Department in the University of Utah in February 2014. We would like to thank the Utah Math Department for its hospitality.

\section{Basic definitions}\label{sec:definition}



\subsection{Log geometry}\label{ss:into-log}
Our construction of very free $\A^1$-curves uses stable log maps \cite{GS, Chen, AC} that are built upon log geometry of Kato-Fontaine-Illusie \cite{KKato}. We briefly recall here some notions in log geometry and stable log maps. Readers who are familiar with stable log maps may skip Section \ref{ss:into-log} and \ref{ss:into-log-map}.

A \emph{log structure} on a scheme $\uX$ is a sheaf of monoids $\cM$ in \'etale topology with a morphism of sheaves of monoids 
$
\alpha:\cM\to \O_{\uX}
$
such that $\alpha$ induces an isomorphism $\alpha^{-1}(\cO^*_{\uX}) \to \cO_{\uX}^*$. Here $\cO_{\uX}$ is viewed as a sheaf of monoids under multiplication. 
For simplicity, We will omit $\alpha$ and refer to $\cM$ as the log structure when there is no danger of confusing. The quotient $\ocM := \cM/\cO_{\uX}^*$ is called the {\em characteristic sheaf} of $\cM$. The pair $X= (\uX, \cM)$ is called a {\em log scheme}. The following example is important to our construction.


\begin{example}\label{ex:toroidal}
Let $\uD \subset \uX$ be a divisor satisfying
\begin{enumerate}
 \item  $X^{\circ} := \uX \setminus \uD$ is smooth, and
 \item for any point $t \in \uD$, there exists an \'etale neighborhood $\uU \to \uX$ of $t$, and a smooth morphism $\uU \to \uT$ with $\uT$ an affine toric variety with the toric boundary $\u\Delta$ such that $\uD|_{U}$ is given by the pull-back of $\u\Delta$. 
\end{enumerate}
 We call such $(\uX, \uD)$ a {\em toroidal pair}, and associate the following divisorial log structure:
 \[
 \cM(\uV) := \{f \in \cO_{\uV} \ | \ f|_{\uV \setminus \uD} \in \cO^*\} \stackrel{\alpha}{\hookrightarrow}  \cO_{\uX}
 \]
for any \'etale morphism $\uV \to \uX$. We call $X = (\uX, \cM)$ the log scheme associated to the toroidal pair $(\uX, \uD)$. \Yi{is toriodal necessary here? later may hard to refer.}
\end{example}

A \emph{morphism} of log schemes $f: Y = (\uY, \cM_Y) \to X = (\uX, \cM_X)$ is given by a pair $(\uf, f^{\flat})$ consisting of a morphism of the underlying schemes $\uf: \uY \to \uX$ and a morphism of log structures $f^{\flat}: \uf^*\cM_{X} \to \cM_{Y}$, where $\uf^*\cM_{X}$ is the pull-back of $\cM_X$  \cite[1.1, 1.4]{KKato}.


{\em Log smoothness} can be defined using the infinitesimal lifting property \cite[(3.3)]{KKato}. A morphism is {\em log smooth} if and only if it is locally toroidal in the sense of \cite[(3.5)]{KKato}. In particular, the log scheme $X$ in Example \ref{ex:toroidal} is log smooth over $\kk$.


\subsection{Stable log maps}\label{ss:into-log-map}
A genus $g$, $n$-marked {\em log curve} over a log scheme $S$ is a pair
$
(\pi: C \to S, \{\sigma_i\}_i)
$
where $(\u{\pi}: \uC \to \uS, \{\sigma_i\}_i)$ is a genus $g$, $n$-marked usual pre-stable curves over $\uS$ with markings $\{\sigma_i\}$, and $\pi$ is a log smooth, integral morphism with local structures described in \cite{FKato, LogCurve}, see also\cite[B.1.1, B.1.2]{Chen}.

We fix a log smooth scheme $X$ as the target. A {\em log map} over $S$ is a morphism of log schemes $f: C \to X$ such that $\pi: C \to S$ is a log curve over $S$. A log map $f$ is called {\em stable} if the underlying morphism $\uf: \uC \to \uX$ is stable in the usual sense. 


Let $\uS$ be a geometric point. We call a stable log map $f$ over $S$ {\em non-degenerate} if the log structure $\cM_S$ is trivial. Geometrically, this means that $f$ has a smooth underlying source curve, and its underlying image does not completely lie on the loci of $X$ with non-trivial log structures.

Consider a stable log map $f: C \to X$ over $S$.  Along a marking $\sigma: \uS \to \uC$, we have an induced morphism of characteristic sheaves
\begin{equation}\label{equ:contact-order}
c_{\sigma}: (\uf\circ\sigma)^*\ocM_X \stackrel{\bar{f}^{\flat}}{\longrightarrow} \sigma^*\cM_{C} = \ocM_S\oplus \NN \longrightarrow \NN
\end{equation}
where the second arrow is the projection to $\NN$. Such morphism $c_{\sigma}$ is called the {\em contact order} along the marking $\sigma$. The contact order along a marking is a deformation invariance of stable log maps. Indeed, they canonically corresponds to the connected components of the evaluation stacks of stable log maps, see \cite{ACGM} and \cite[Section 5.2]{log-bound}.

Denote by $\Gamma = (g,n,\{c_i\}, \beta)$ the collection of numerical data where $g$ is the genus, $n$ is the number of markings, $c_i$ is the contact order at the $i$-th marking, and $\beta$ is the curve class. It was proved in \cite{GS, Chen, AC, Loghom-wise} that the stack of stable log maps with numerical data $\Gamma$ is represented by an algebraic stack carrying a minimal log structure. We may thus study stable log maps using their deformation theory.  

Consider a stable log map $f: C \to X$ over $S$. By \cite[5.9]{Logcot} the deformation and obstruction theory of $f$ are controlled by $H^0(f^*T_{X})$ and $H^1(f^*T_{X})$ respectively. Here the {\em log tangent bundle} $T_X = \Omega_{X}^{\vee}$ is defined as the dual of the log cotangent bundle of $X$ \cite[1.7 and 3.10]{KKato}. In particular, suppose $\uS$ is a geometric point.  If $H^1(f^*T_{X})=0$, then $f$ can be deformed to a non-degenerate stable log map.  We refer to \cite[2.5]{Chen} for more details of the deformation of log maps.



\subsection{Log rational curves}\label{ss:A1-connectedness}



Recall that $\P^1_\infty$ is the log scheme associated to the pair $(\P^1,\{\infty\})$. Let $X$ be a log scheme, and $X^{\circ} \subset X$ the loci with the trivial log structure. By the following lemma, an $\A^1$-curve on a log variety $X$ is the same as an $\A^1$-curve on $X^{\circ}$. 

\begin{lemma}\label{lem:A1-interior}
Assume that $X$ is proper. Then for any morphism $\uf^{\circ}: \A^1 \to X^{\circ}$, there is a unique stable log map $f: \P^{1}_{\infty} \to X$ such that the following diagram is commutative
\[
\xymatrix{
\A^1 \ar[r]^{\uf^{\circ}} \ar[d] & X^{\circ} \ar[d] \\
\P^{1}_{\infty} \ar[r]^{f} & X
}
\]
where the vertical arrows are the corresponding embeddings. Furthermore, there are precisely two possibilities:
\begin{enumerate}
 \item if the contact order of $f$ at $\infty$ is non-trivial, then $\uf^{\circ}$ is an $\A^1$-curve;
 \item if the contact order of $f$ at $\infty$ is trivial, then the underlying morphism $\uf: \PP^1 \to \uX$ of $f$ factors through $X^{\circ}$. 
\end{enumerate}
\end{lemma}
\begin{proof}
Since the underlying scheme $\uX$ is proper, the underlying morphism $\uf$ of $f$ is uniquely determined by $\uf^{\circ}$. To determine $f$, it remains to construct a morphism of sheaf of monoids $f^{\flat}: \uf^{*}\cM_X \to \cM_{\PP^1_{\infty}}$ that fits in the following commutative diagram
\[
\xymatrix{
\uf^{*}\cM_X \ar[rr]^{f^{\flat}} \ar[rd]_{\alpha_1} && \cM_{\PP^{1}_{\infty}} \ar[ld]^{\alpha_2} \\
 & \cO_{\PP^1} &
}
\]
where the two arrows $\alpha_1$ and $\alpha_2$ are the structure arrows of the corresponding log structures. Since the image of $\uf|_{\AA^1} = \uf^{\circ}$ lands in the loci $X^{\circ}$ with the trivial log structure, the image of $\alpha_1$ consists of functions that are invertible away from $\infty$. This implies that $\alpha_1(\uf^*\cM_{X}) \subset \cM_{\P^1_{\infty}} \subset \cO_{\PP^1}$, hence the unique morphism $f^{\flat}$. 

If the contact order of $f$ at $\infty$ is non-trivial, $\uf^*\cM_{X}$ is non-trivial along $\infty$. This implies that $\uf(\infty) \in X\setminus X^{\circ}$. Hence $\uf^{\circ}$ is an $\A^1$-curve.

Otherwise, the contact order of $f$ at $\infty$ is trivial, and the morphism $f^{\flat}$ factors through $\cO^*_{\P^1}$. Hence the log structure $\uf^*\cM_{X}$ is the trivial one, and $f$ factors through $X^{\circ}$.
\end{proof}

Similar to Definition \ref{def:intro-uniruled-connected}, we have the logarithmic version:

\begin{definition}\label{def:uniruled-connected}
Let $X$ be a proper, log smooth variety of dimension $n$. 

\vspace{1mm}

\noindent
$\bullet$ 
$X$ is called \emph{log uniruled} (resp. \emph{separably log uniruled}) if there is a  scheme $\uT$ of dimension $n-1$ and a log morphism 
\begin{equation}\label{equ:uniruled}
f: \P^1_\infty \times \uT\to X
\end{equation}
which is dominant (resp. dominant and separable). $X$ is called {\em $\A^1$-uniruled} (resp. {\em separably $\A^1$-uniruled}), if furthermore $f$ is a family of $\A^1$-curves.


\vspace{1mm}

\noindent
$\bullet$ $X$ is called \emph{log rationally connected} (resp. \emph{separably log rationally connected}) if there is a scheme $\uT$ and a log morphism 
$
f:\P^1_\infty\times \uT\to X
$
such that the morphism   
\begin{equation}\label{equ:rc1}
	\begin{split}
	f^{(2)}:	\P^1_\infty\times\P^1_\infty \times  \uT &\to X\times X\\
	(t_1,t_2,y) &\mapsto (f(t_1,y),f(t_2,y))
	\end{split}
\end{equation}
	is dominant (resp. dominant and separable). $X$ is called {\em $\A^1$-connected} (resp. separably $\A^1$-connected) if we further require $f$ to be a family of $\A^1$-curves.
\end{definition}

%

\begin{remark}
The definition of log uniruledness and log rationally connectedness are compatible with Campana's orbifold uniruledness and rational connectedness for boundary with infinite weight \cite[Definition 9.4]{Campana-orb-rat}.
\end{remark}

We observe that the $\A^1$-uniruledness and $\A^1$-connectedness are intrinsic to the open loci of the log variety with the trivial log structure.

\begin{proposition}\label{prop:open-A1-uniruled}
Let $X$ be a proper log smooth variety. Denote by $X^{\circ} \subset X$ the loci with the trivial log structure. The following are equivalent:
\begin{enumerate}
 \item $X$ is $\A^1$-uniruled (resp. separably $\A^1$-uniruled). 
 \item $X^\circ$ is $\A^1$-uniruled (resp. separably $\A^1$-uniruled).  
\end{enumerate}
\end{proposition}

\begin{proposition}\label{prop:open-A1-connected}
Notations as in Proposition \ref{prop:open-A1-uniruled}, the following are equivalent:
\begin{enumerate}
 \item $X$ is $\A^1$-connected (resp. separably $\A^1$-connected). 
 \item $X^\circ$ is $\A^1$-connected (resp. separably $\A^1$-connected). 
\end{enumerate}
\end{proposition}

We will provide a proof of Proposition \ref{prop:open-A1-connected}. The proof of Proposition \ref{prop:open-A1-uniruled} is similar, and is left to the reader. 

\begin{proof}[Proof of Proposition \ref{prop:open-A1-connected}]
Since contact orders are deformation invariant, over a connected family of stable log maps, the contact order along a fixed marking can be either trivial over the whole family, or non-trivial over the whole family.  The direction ``(1) $\Rightarrow$ (2)'' follows directly from Definition \ref{def:uniruled-connected} by removing the marking $\infty$. 

Conversely, replacing $\uT$ by its normalization, we may assume that $\uT$ is normal. The morphism $f^{\circ}$ induces a rational map 
\[
\xymatrix{
\P^1 \times \uT \ar@{-->}[r] & \uX.
}
\]
Replacing $T$ by an open dense subscheme which is again denoted by $\uT$, we may assume that the above rational map is a morphism:
\[
\uf: \P^1 \times \uT \to \uX. 
\]
Furthermore, since $f^{\circ}_t$ is proper for any $t \in \uT$, the image of the marking $\infty \times \uT$ via $\uf$ is contained in $\uX \setminus X^{\circ}$. A similar argument as for Lemma \ref{lem:A1-interior} implies that $\uf$ can be lifted to a unique log morphism $f: \P^{1}_{\infty}\times \uT \to X$. This provides the family (\ref{equ:rc1}) as needed in Definition \ref{def:uniruled-connected}.
\end{proof}


\begin{proposition}\label{prop:A1-uniruled}
	Let $X$ be a proper log smooth variety. Then we have
	\begin{enumerate}
		\item 
		$X$ is separably log uniruled if and only if there exists a free log rational curve on $X$; 
		\item $X$ is separably log
		rationally connected if and only if there exists a very free log rational curve on $X$. 
	\end{enumerate}	
\end{proposition}

\begin{proof}
The proof of this proposition is similar to the case of rational curves \cite[IV 1.9, 3.7]{kollar}. For completeness, we sketch the proof of (2). The proof of (1) is similar, and is omitted here.

Denote by $\cA:=\cA_{2}(X)$ the log stack of genus zero stable log maps with a unique marking $\infty$. Let $\cA^{\circ} \subset \cA$ be the open substack with the trivial log structure. Denote by 
$
\cF: \cC \to X, 
$
the universal stable log map over $\cA^{\circ}$. For a log rational curve $[f] \in \cA^{\circ}$, Let $\uT \to \cA^{\circ}$ be a smooth morphism with image containing $[f]$, such that the pull-back of $\cF$ over $\uT$ is given by 
\[
F : \P^1_{\infty}\times \uT \to X.
\]

Assume that $[f] \in \uT$.  Consider the induced morphism:
\[
F^{(2)}: \P^1_{\infty}\times\P^1_{\infty}\times \uT \to X.
\]
Consider two geometric points $p,q \in \A^1 = \P^1_{\infty} \setminus \{\infty\}$.  We calculate that the log differential $\diff F^{(2)}$ at $(p,q,[f])$ is of the form  
\[
\diff F^{(2)}(p,q,[f]):(df(p)+\phi(p,f),df(q)+\phi(q,f)), 
\]
where $df(s)$ is the log differential of $f$ at $s$ and $\phi(s,f)$ is the natural evaluation
\[
\phi(s,f):H^0(\P^1,f^*TX\otimes \cO_{\P^1}(-s))\to f^*TX\otimes \kk(s).
\]

Following the same argument of \cite[II 3.5]{kollar}, we observe that $dF^{(2)}(p,q,[f])$ is surjective if and only if the following is also surjective 
\[
\phi(p,q,f)=(\phi(p,f),\phi(q,f)):H^0(\P^1,f^*TX)\to f^*TX\otimes \kk(p)\oplus f^*TX\otimes \kk(q).
\]
Indeed, since the log tangent bundle of $T_{\P^1_\infty}$ is the line bundle $\O_{\P^1}(1)$, the image of the morphism $(df(p),df(q))$ is contained in the image of $\phi(p,q,f)$. This implies that $dF_2$ is surjective if and only if $\phi(p,q,f)$ is surjective. Observe that the latter condition is equivalent to that $f^*TX$ is positive.
	
Since separable log rational connectedness is equivalent to $F^{(2)}$ is dominant and generic smooth, it is equivalent to the existence of a very free log rational curve as well.
\end{proof}

The following observation is a generalization of \cite[IV 1.11 and 3.8]{kollar}.

\begin{corollary}\label{cor:van}
	Let $X$ be a proper log smooth variety. 
	\begin{enumerate}
		\item If $X$ is separably log uniruled, then $H^0(X, K_X^m)=0$ for every $m>0$.
		\item If $X$ is separably log rationally connected, then $H^0(X, (\Omega^1_X)^{\otimes m})=0$ for every $m>0$.
	\end{enumerate}	
\end{corollary}
\begin{proof}
Here we only verify (2). The proof of (1) is similar and is left to the reader. By Proposition \ref{prop:A1-uniruled}, there exists a log rational curve $f:\P^1_\infty\to X$ such that $f^*T_X$ is positive. Thus $f^*\Omega^1_X$ is the sum of line bundles of negative degree. Therefore, any section of $(\Omega^1_X)^{\otimes m}$ vanishes along $f(\P^1)$. Such very free log rational curves cover a dense open subset of $X$. 
\end{proof}

\section{Curves in the center}\label{sec:DF-var}

\subsection{Centers of Deligne-Faltings type}
Let $X$ be a log scheme. \R{Throughout this paper, we will always assume that $\uX$ to be quasi-compact.} By \cite[Lemma 3.5]{LogStack} there is a canonical stratification $\{X_{\lambda}\}_{\lambda \in \Lambda}$ associated to $X$ such that
\begin{enumerate1}
 \item $X_{\lambda} \to X$ is a connected locally closed subscheme with the pull-back log structure.
 \item The sheaf of groups $\ocM_{X_{\lambda}}^{gp}$ is a locally constant sheaf.
 \item $X = \cup_{\lambda}X_\lambda$ is a disjoint union.
\end{enumerate1}
Denote by $\overline{X}_{\lambda}$ the closure of $X_{\lambda}$ in $X$. $X_{\lambda}$ is called a {\em center} of $X$ if $X_{\lambda} = \overline{X}_{\lambda}$. 

For later use, we reserve the letter $Y$ for a connected center of $X$, and view $Y$ as a log scheme with the log structure pulled back from $X$. We observe that when $X$ is log smooth, $Y$ also has smooth underlying structure, see \cite[Lemma 3.5(ii)]{LogStack}.

\begin{definition}\label{def:DF-center}
A center $Y$ of a log scheme $X$ is called of {\em Deligne-Faltings type} \R{
if there is a fine, saturated, sharp monoid $P$, and a global morphism of sheaves of monoids $P_X \to \ocM_{Y}:= \cM_{Y}/\cO^*_{Y}$ which \'etale locally lifts to a chart of $\cM_{Y}$. Here we view $P_X$ as the global constant sheaf of monoids with coefficients in the $P$. For simplicity, we may identify $P$ with $P_X$ if there no danger of confusion.
}

\end{definition}

\R{

\begin{lemma}\label{lem:dis-chart-DFcenter}
Let $Y$ be a center of some log scheme $X$. Then the following statements are equivalent:
\begin{enumerate}
 \item The sheaf of monoids $\ocM_{Y}$ is globally constant.
 \item The natural morphism $\beta: \Gamma(\uY, \ocM_{Y}) \to \ocM_Y$ is a chart.
  \item $Y$ is of Deligne-Faltings type.
\end{enumerate}
The chart $\beta$ in (2) is called the {\em distinguished chart} of $\cM_{Y}$.
\end{lemma}
\begin{proof}
The equivalence between (1) and (2) is obvious. Clearly, (2) implies (3). It remains to show that (3) implies (1).  Consider the following diagram
\[
\xymatrix{
\Gamma(\uY, \ocM_Y) \ar[rr] \ar[d] && \ocM_Y \ar[d] \\
\Gamma(\uY, \ocM_Y)^{gp} \ar[rr] && \ocM_{Y}^{gp}.
}
\]
Since the vertical arrows are inclusions of sheaves of monoids, it suffices to show that $\ocM_{Y}^{gp}$ is globally constant. Note that $\ocM_{Y}^{gp}$ is locally constant on $\uY$. For any point $y \in Y$, consider the specilization morphism
\[
\phi_y: \Gamma(\uY, \ocM_Y)^{gp} \to \ocM_{Y,y}^{gp}.
\]
Since $Y$ is of Deligne-Faltings type,  one observe that $\phi_y$ is surjective for any $y$. Consider an element $\delta \in \ker \phi_y$. Then we have $\delta = a - b$ for some elements $a, b \in \Gamma(\uY, \ocM_Y)$. Thus, we have $\phi_y(a) = \phi_y(b)$ in $\ocM_{Y,y}^{gp}$. Since $\ocM_{Y}^{gp}$ is locally constant on $\uY$ this implies that $a = b$ as a section in $\ocM_Y$ in a neighborhood of $y$ in $\uY$. Since $\uY$ is connected by the definition of center, $a = b$ over $\uY$. This implies that $\delta = 0$ and $\phi_y$ is an isomorphism for any $y \in Y$. Thus, the sheaf $\ocM_{Y}^{gp}$ is globally constant.
\end{proof}
}

The following observation follows by considering a toroidal modification:

\begin{lemma}\label{lem:center-resolution}
Let $Y$ be a center of a log smooth variety $X$. Let $\pi: X' \to X$ be a \R{proper}, birational, log \'etale morphism, and $Y' \subset X'$ be a center over $Y$. 
Then 
\begin{enumerate1}
 \item $(\bar{\pi}^{\flat})^{gp}|_{Y}: \pi^*\ocM^{gp}_{Y} \to \ocM^{gp}_{Y'}$ is an isomorphism of sheaves of groups.
 \item The underlying schemes $\uY$ and $\uY'$ are smooth in the usual sense.
 \item The underlying morphism $\underline{\pi}|_{\uY'}: \uY' \to \uY$ is finite and \'etale.
\end{enumerate1}
\end{lemma}
\R{
\begin{proof}
Since the statements are local around $Y$, by the log smoothness of $X$, we may assume that $\uX = \spec R[P]$ for some ring $R$, and a fine, saturated, and sharp monoid $P$, with the log structure $\cM_X$ given by the log structure associated to the pre-log structure $P \to R[P]$, see \cite[(1.3)]{KKato}. Then the birational log \'etale morphism $\pi': X' \to X$ is given by toroidal modifications of the underlying toric varieties. Denote by $U \subset X'$ the open log subscheme containing the center $Y'$ over $Y$. We may assume that $\uU = \spec R[P']$ with the log structure $\cM_U$ given by the log structure associated to the pre-log structure $P' \to R[P']$. The morphism $U \to X$ is induced by the morphism $P' \to P$ on the level of monoids. 

Notice that $\ocM_{Y'} = P'$ and $\ocM_{Y} = P$ are both globally constant sheaves of monoids with the morphism $(\bar{\pi}^{\flat})^{gp}|_{Y}: \pi^*\ocM^{gp}_{Y} \to \ocM^{gp}_{Y'}$ given by the morphism $P \to P'$ induced by the toroidal modifications. This proves (1). 

The log smoothness of $X$ implies that $\spec R$ is smooth in the usual sense, see \cite[Theorem (3.5)]{KKato}. Furthermore, in the local case the projection $\uY' = \spec R \to \uY = \spec R$ is just the identity. This proves (2) and (3).  
\end{proof}
}

In this paper, we are particularly interested in the following situation.

\begin{lemma}\label{lem:center-DF}
Let $X$ be a log smooth variety, and $Y \subset X$ \R{be} a center of $X$. Assume $\uY$ is proper and separably rationally connected. Then $\ocM_{Y}$ is a globally constant sheaf of monoids over $\uY$. In particular the center $Y$ is of Deligne-Faltings type.
\end{lemma}
\begin{proof}
Let $X' \to X$ be a birational log \'etale morphism such that $X'$ has simple normal crossings boundary. Such resolution exists over algebraically closed \R{fields} of arbitrary \R{characteristic}, \R{see \cite[Theorem 5.10]{Ni} or \cite[Corollary 4.5.4]{log-bound}}. Let $Y' \subset X'$ be a center over $Y$. Since $X'$ has simple normal crossings boundary, the sheaf of free monoids $\ocM_{Y'} \cong \NN^k$ is a globally constant sheaf of monoids for some positive integer $k$.

By \cite[Corollary 3.6]{Debarre}, the separably rationally connectedness of $\uY'$ implies the underlying morphism $\uY' \to \uY$ is an isomorphism. Since $\ocM_{Y'}$ is a constant sheaf of \R{monoids}, by Lemma \ref{lem:center-resolution}, the sheaf of groups $\ocM_{Y}^{gp}$ is a constant sheaf of monoids. Then the canonical inclusion $\ocM_{Y} \subset \ocM_{Y}^{gp}$ implies that $\ocM_{Y}$ is globally constant sheaf of monoids. 
\end{proof}

\subsection{Admissible curve classes in Deligne-Faltings centers}\label{ss:curves-on-DF}

\begin{proposition}\label{prop:DF-line-bundle}
Let $X$ be a log scheme, and $\gamma: P \to \ocM_{X}$ be a morphism from a globally constant sheaf of monoids $P$.  Then there exists a natural map of monoids
\begin{equation}\label{equ:DF-line-bundle}
L_{\gamma}: P \to \pic(\uX)
\end{equation}
given by 
$
\delta  \to (q^{-1}(\gamma(\delta))^{\vee},
$
where $q: \cM_{X} \to \ocM_{X}$ is the quotient morphism. 
\end{proposition}
\begin{proof}
Note that $(q^{-1}(\gamma(\delta))$, hence its dual is an $\cO_{\uX}^*$-torsor, which corresponds to a unique element in $\pic(\uX)$. 
\end{proof}

\begin{definition}\label{def:log-one-cycle}
Let $Y$ be a Deligne-Faltings center, and 
$
\beta: P := \Gamma(\uY, \ocM_{Y}) \to \ocM_{Y}
$ 
be the distinguished chart.  The \emph{total contact order} of a curve class $F \in N_1(\uY)$ is the element $\co(F)\in (P^{gp})^\vee$ defined by 
\begin{equation}\label{equ:curve-co}
\co(F)(\delta) = c_1(L_{\beta}(\delta)) \cdot F, \ \ \ \mbox{for any }\delta \in P^{gp}
\end{equation}
where $c_1(L_{\beta}(\delta))$ is the first chern class of $L_{\beta}$. A curve class $F$ is \emph{log admissible} if $\co(F) \in P^\vee$. 
\end{definition}


\R{

\begin{example}\label{ex:SNC}
We given an example of total contact orders in the case of simple normal crossings boundaries.  Let $\uX$ be a smooth variety, and $\uD = \cup_{i =1}^{n} \uD_i \subset \uX$ be a simple normal crossings divisor, where $\uD_i$ is the irreducible component of $\uD$ for each $i$. Denote by $X = (\uX, \cM_X)$ the log scheme associated to the pair $(\uX, \uD)$, see Example \ref{ex:toroidal}. Assume $X$ has a unique center $Y$ with $\uY = \cap_{i=1}^{n} \uD_i$.

By \cite[Complement 1]{KKato} and \cite[A.2]{Chen}, the characteristic sheaf $\cM_{X}$ admits a morphism 
$\beta_X: P = \NN^n \to \ocM_{X}$
such that $\beta_X$ locally lifts to a chart of $\cM_X$. We give a local description of $\beta_X$ as follows.

Let $\delta_1, \cdots, \delta_n$ be a basis of $P$. For each point $x \in X$, let $s_i = 0$ be the local equation of $\uD_i$ around $x$ for each $i$. Then we may view $s_i$ as a local section of $\cM_X$. Denote by $\bar{s}_i$ the image of $s_i$ in $\ocM_{X}$. Then locally around $x$, the morphism $\beta_X$ is given by 
\begin{equation}\label{equ:SNC-char-local}
\beta_X(\delta_i) = \bar{s}_i, \ \ \ \mbox{for each }i.
\end{equation}
In this case, the induced morphism
$
\beta_Y := (\beta_X)|_Y: P \to \ocM_Y
$
is the distinguished chart. Further notice that $L_{\beta_Y}(\delta_i) = \cO_{\uX}(D_i)|_{\uY}$. Thus an effective curve $F$ in $\uY$ is log admissible if and only if 
$
\co(F)(\delta_i) = [F] \cdot [\uD_i] \geq 0 \ \  \mbox{for all } i.
$
\end{example}
 }

\begin{lemma}\label{lem:curve-in-center}

Suppose $Y$ is a Deligne-Faltings center of a log smooth variety $X$. Let $\pi: X' \to X$ be any proper, birational, log \'etale morphism, and $Y' \subset X'$ be a center over $Y$. \R{Then $Y'$ is also of Deligne-Faltings type with $\uY' = \uY$. } Consider the distinguished charts $\beta: P:= \Gamma(\uY, \ocM_{Y}) \to \ocM_{Y}$ and $\beta': Q := \Gamma(\uY', \ocM_{Y'}) \to \ocM_{Y'}$. For any effective curve $\uC \subset  \uY$, we have 
$
\co_{Y}([\uC]) = \co_{Y'}([\uC])
$ 
in $(Q^{gp})^{\vee} = (P^{gp})^{\vee}$, \R{where $\co_{Y}$ and $\co_{Y'}$ are the total contact orders associated to $Y$ and $Y'$ respectively, see (\ref{equ:curve-co}).}


\end{lemma}
\begin{proof}
Consider the following commutative diagram over $\uY'$:
\begin{center}
\begin{equation}
\xymatrix{
\pi^*\cM_{Y} \ar[r] \ar[d] & \cM_{Y'} \ar[d] \\
\pi^*\ocM_{Y} \ar[r] & \ocM_{Y'} \\
P \ar[u]^{\cong} \ar[r] & Q \ar[u]^{\cong}.
}
\end{equation}
\end{center}
We want to determine $c_1(L_{\beta}(\delta))\cdot[\uC]$ for each $\delta \in Q$. \R{By Lemma \ref{lem:center-resolution}(1)}, the bottom arrow induces an isomorphism $P^{gp} \cong Q^{gp}$. The statement then follows.
\end{proof}

The following definition is crucial in our construction of $\A^1$-curves.

\begin{definition}\label{def:full-free}
\R{A Deligne-Faltings type center $Y \subset X$ is called {\em fully free (resp. primitively free)} if there exists free rational curves $\uR_1,\cdots,\uR_l$ on $\uY$ such that their total contact orders lie in $P^\vee$, and span the \R{$\kk$-}vector space $(P^{gp})^{\vee}\otimes_{\ZZ}\kk$ (resp. the lattice $(P^{gp})^{\vee}$).  }
\end{definition}

\R{
\begin{remark}
Considering the $\kk$-vector space spanned by the total contact orders will be crucial in the construction of very free $\AA^1$-curves for base fields of arbitrary characteristic, see Lemma \ref{lem:attach-curve-key}.
\end{remark}
}

\R{
\subsection{Examples of simple normal crossings boundaries}
We next provide some interesting examples of fully free centers.

\begin{lemma}\label{lem:TZ}
	Let $\uZ$ be a proper, smooth, separably rationally connected variety. There exist very free rational curves $\uR_0, \uR_1, \cdots, \uR_l$ on $\uZ$ whose curve classes span the $\ZZ$-lattice $N_1(\uZ)$. 
\end{lemma}

\begin{proof}By \cite[Theorem 1.3]{Tian-Zong}, $\CH_1(\uZ)$ is generated by rational curves. In particular, there exists rational curves $\uR'_1, \cdots, \uR'_{l}$ on $Z$ whose curve classes span $N_1(\uZ)$. Now we pick a very free curve $\uR_0$, expressed as a linear combination in $N_1(\uZ)$:
\[
[\uR_1]= a_1[\uR'_1]+\cdots+a_l[\uR'_{l}].
\]


For any $i= 1, \cdots, l$, let $N_i$ be a positve integer such that the class $N_i[\uR_1]+ [\uR'_{i}]$ can be represented by a very free rational curve $\uR_i$. The existence of such very free rational curves follows from \cite[Theorem II.7.9, Proposition II.7.10]{kollar}. Now rational curves $\uR_0,\cdots, \uR_l$ are very free, whose curve classes span  $N_1(\uZ)$.
\end{proof}

\begin{proposition}\label{prop:new-fully}
Suppose $\ch\kk =0$. Let $\uX$ be a smooth projective variety, and $\uD=\uD_1+\cdots+\uD_r$ be a simple normal crossings divisor on $\uX$. Let $X$ be the log smooth variety associated to the pair $(\uX,\uD)$. Assume that 
\begin{enumerate}
 \item $\uD_i$ is irreducible and ample for all $i$;
 \item  the numerical equivalence classes of $\uD_1,\cdots,\uD_r$ are linearly independent in the N\'eron-Severi group $N^{1}(\uX)_{\QQ}$;
		\item the center $\uY:= \cap_{i=1}^r \uD_i$ is rationally connected of dimension $\ge 2$.
	\end{enumerate}
Then $\uY$ is a fully free center.
\end{proposition}

\begin{proof}
Since $\uY$ is rationally connected, by Lemma \ref{lem:TZ} we have very free rational curves $R_0, \cdots, R_l$ on $\uY$ whose curve classes span $N_1(\uY)$.  Let $L_i$ be the pullback of the Cartier divisor $\O_{\uX}(\uD_i)$ to $\uY$. By Example \ref{ex:SNC}, the total contact order of $\uR_i$ is given by the vector of intersection numbers
\[
(\uD_1\cdot \uR_i,\cdots, \uD_r\cdot\uR_i).
\] 
The center is fully free if and only if the matrix $M= (\uD_j \cdot \uR_i)$ satisfies:
	\begin{enumerate}
		\item every entry of $M$ is nonnegative;
		\item the rank of $M$ is $r$.
	\end{enumerate}

Part (1) follows from the ampleness of $\uD_i$. By Lefschetz hyperplane theorem, the Cartier divisors $L_1, \cdots, L_r$ are linearly independent in $N^{1}(\uY)_{\QQ}$. Part (2) then follows from the fact that $\uR_0, \cdots, \uR_l$ span $N_1(\uY)$.
\end{proof}
}



\section{Log comb construction}\label{sec:comb}

\R{
We next introduce a comb construction in the log setting. Unlike the situation in \cite{CZ}, the underlying stable map of the comb under consideration will be contained in a center. In general, a stable map can not be lifted to a stable log map.  Section \ref{ss:underlying-comb}, \ref{ss:mon-obs}, and \ref{ss:lift-to-log} is devoted to show that the configuration of stable maps in our situation can be lifted. We analyze the deformation of such stable log map in Section \ref{ss:comb-def}, and then prove Theorem \ref{thm:comb-smoothing} by constructing an unobstructed stable log map whose general smoothing produces a very free $\AA^1$-curve.

}

\subsection{The underlying comb}\label{ss:underlying-comb}
We fix some notations that will be used throughout this section.  Let $X$ be a log smooth scheme with simple normal crossing boundary $\uD$ given by smooth irreducible components $\uD_{1},\cdots, \uD_k$. For any $\lambda \subset [k] := \{1,\cdots, k\}$, denote by $\uD_{\lambda} = \cap_{i \in \lambda} \uD_{i}$.  Let $D, D_i,$ and $D_{\lambda}$ be the corresponding log schemes with the log structure restricted from $X$.

\begin{definition}\label{def:underlying-comb}
Assume $D_{\lambda} \subset X$ is a center. Write $n := |\lambda|$. A usual stable map $\uf: \uC \to \uD_{\lambda}$ is called a \R{{\em pre-log comb}} if
\begin{enumerate1}
 \item $\uC =\uC_0 \cup \PP_1 \cup \cdots \cup \PP_m$ is a prestable curve with a unique marking $q_\infty \in \uC_0$.

 \item $\uC_0$ is a smooth irreducible curve of genus $g$, and $\PP_i$ is an irreducible rational curve attached to $\uC_0$ at a unique node $r_i$.

 \item $c_{ij}:= \uf_*[\PP_i] \cdot \uD_j \geq 0$ for each $i = 1,\cdots, k$ and $j \in \lambda$. Write $\vec{c}_{i} = (c_{i1}, \cdots, c_{in})$.

 \item $c_{\infty j} := \uf_*[\uC] \cdot \uD_j \geq 0$ for all $j \in \lambda$. Write $\vec{c}_{\infty} = (c_{\infty 1},\cdots, c_{\infty n})$.
\end{enumerate1}
\R{
A {\em log comb} in $D_{\lambda}$ is a stable log map whose underlying map is a pre-log comb.
}
\end{definition}

\begin{remark}
For future applications, we include the case with higher genus handle $\uC_0$. This will not complicate our discussion below. 
\end{remark}

\subsection{Combinatorial obstructions}\label{ss:mon-obs}

Our next goal is to understand when a pre-log comb $\uf$ can be \R{lifted to} a stable log map to $D_{\lambda}$, hence to $X$. The first obstruction to have such lifting is on the level of characteristic monoids, which is known as admissibility in the rank one case \cite[Definition 3.3.6]{Chen}. 

Denote by $P := \oplus_{j \in \lambda} \NN_j$ \R{with $\NN_j = <\delta_j> \cong \NN$}. Since \R{the boundary of} $X$ has simple normal crossings, the log structure $\cM_{X}$ is of Deligne-Faltings type, \R{and so is $\cM_{\lambda}$}. By Lemma \ref{lem:dis-chart-DFcenter}, there is a distinguished chart
\begin{equation}\label{equ:SNC-DF}
\beta: P \to \ocM_{\lambda}.
\end{equation}
We further assume that each copy of $\NN_j$ corresponds to the log structure from $D_j$, see Example \ref{ex:SNC}. Write \R{$\co_i := \co(\PP_i) \in (P^{gp})^{\vee}$} for each $i$, see (\ref{equ:curve-co}).  Consider $\{\delta_1,\cdots, \delta_n\}$ the set of standard generators of $P$. Denote by
\begin{equation}\label{equ:intersection}
c_{ij} = \co_i(\delta_j).
\end{equation}

\begin{notation}
To proceed further, we introduce the following:
 \begin{enumerate1}
 \item For each $\PP_i$ \R{(resp. $\uC_0$)} we introduce $P_i \cong P$ \R{(resp. $P_0 \cong P$)} with the set of standard generators 
 \[
 \{e_{i1}, \cdots, e_{in}\}  \ \ \ \mbox{for }  i=0,1,\cdots, m.
 \] 
 Here $e_{ij}$ is the vertex element corresponding to $D_j$ as in \cite[4.1.2]{AC}.

 \item For each node $r_i$ we introduce $N_i = \langle l_i \rangle \cong \NN$. Here $l_i$ is the edge element of the node $r_i$ as in \cite[4.1.2]{AC}.

 \item For each $i \in \{1,\cdots, m\}$, we introduce a morphism of monoids \Yi{explain P0, 3.54}\Qile{Add an explanation above.}
\[\phi_i := id\oplus \co_i : P_{0} \to P_{i}\oplus N_i \] 
given by 
$\phi_i(\delta) = \delta + \co_i(\delta)\cdot l_i $
for $\delta \in P_0 \cong P_i$. This is given by the edge equation in \cite[4.1.2]{AC}.
 \end{enumerate1}
\end{notation}

Note that the sets of monoids $\{P_0, P_1\oplus N_1,\cdots, P_{m}\oplus N_m\}$ and morphisms $\{\phi_i\}$ form a finite system in the category of monoids, denoted by $\Phi$. 

\R{
Denote by $\Phi^{gp}$ the groupification of the the finite system of $\Phi$ consisting of group homomorphisms $\{\phi_{i}^{gp}: P^{gp}_{0} \to P_{i}^{gp}\oplus \NN^{gp}_{i}\}$. Denote by $G = \indlim \Phi^{gp}$ the colimit of $\Phi^{gp}$. By the definition of colimits, we have the natural morphism 
\[
P_0 \oplus \sum_{i=1}^{m}(P_{i}\oplus N_i) \to G.
\]
Denote by $\ocN$ the saturated submonoid in $\ocM^{gp}$ generated by the image of the above morphism, and by $\ocN^* \subset \ocN$ the maximal subgroup of $\ocN$. Let $\ocM = \ocN / \ocN^*$. Using the universal property of colimits, we observe that
}
\begin{equation}\label{equ:min-monoid}
\ocM := \indlim \Phi
\end{equation}
where the colimit is taken in the category of fine, saturated and sharp monoids. 
 Indeed, one checks that $\ocM$ is the minimal monoids constructed in \cite[4.1.2]{AC}. 

\R{
\begin{remark}
By \cite[Lemma 3.3(2)]{irr-rat}, there is a natural splitting $\ocM^{gp} = P^{gp}\oplus \sum_{i=1}^{m}\ZZ$. However, the monoid $\ocM$ is in general not the direct sum $P \oplus \sum_{i=1}^{m}\N$. \Yi{3.55: Qile: The referee is wrong about this.} The colimit description as above is helpful for proving the properties in Lemma \ref{lem:graph-obs} without calculating the precise formation of $\ocM$. 
\end{remark}
}

By the formulation, we have the natural maps
\[
\chi_i: P_i \to \ocM, \ \ \mbox{and} \ \  \Theta_i: N_i \to \ocM.
\]

\begin{lemma}\label{lem:graph-obs}
Both $\chi_i^{-1}(0)$ and $\Theta_i^{-1}(0)$ are the trivial monoid for any $i$. 
\end{lemma}
\begin{proof}
Consider the set of morphisms
\[
g_0 = id_P : P_0 \to P \ \ \mbox{and} \ \ g_i := id\oplus 0: P_{i}\oplus N_i \to P
\]
where $g_i$ is the projection to its first factor for $i \neq 0$. Thus the set $\{g_i\}_{i=0}^{m}$ induces a morphism from the system $\Phi$ to $P$. Since $P$ is a fine, saturated and \R{sharp monoid}, we obtain a morphism
$
g: \ocM \to P.
$
By the choice of $g_i$, for any $\delta \in P_i$ we have $g_i(\delta) = 0$ if and only if $\delta = 0$ in $P_i$. This implies that $\chi_i^{-1}(0)$ \R{is trivial} for any $i$.\Yi{incomplete 3.58: Qile: fixed!}

For each $i \neq 0$, we introduce a set of morphisms \R{
\[
h_{i0} := \co_i: P_0 \to \NN, \ \ \ \ h_{ii} := 0 \oplus id: P_i \oplus N_i \to \NN
\]
and
\[
h_{ij} := \co_i \oplus 0: P_j \oplus N_j \to \NN, \ \ \ \mbox{for $j\neq i$}.
\]
}
\R{Since $P_0 = P_j = P$, the above morphisms are well-defined.  Note that the set of morphisms $\{h_{ij}\}_{j=0}^{m}$ induces a morphism from the system $\Phi$ to $\NN$,} hence a morphism $\ocM \to \NN$. Furthermore, we check the composition $N_i \to \ocM \to \NN$ is the identity. This proves that $\Theta_i^{-1}(0)$ is trivial \R{for $i = 1, 2, \cdots, m$}.\Yi{3.59: Qile: Problem fixed!}
\end{proof}

\begin{remark}\Yi{expand, 3.60: Qile:expanded!}
If $\uf$ can be lifted to a stable log map $f: C/S \to X$, then the base monoid $\ocM_{S}$ will automatically satisfy the conditions in Lemma \ref{lem:graph-obs}. \R{
In fact, the image of the generator $1 \in N_i$ under $\theta$ corresponds to infinitesimally the smoothing parameter of the node $r_i$. The log lift $f$ has a node $r_i$ if and only if $\theta(1) \neq 0$. 

The image of the element $e_{ij} \in P_i$ under $\chi_i$ is called the {\em degeneracy} of the component $\PP_i$. If the lift $f$ exists, then the image of $\PP_i$ under $\uf$ lying in $D_j$ if and only if $\chi_i(e_{ij}) \neq 0$. 
} 

Such property of $\ocM$ is called {\em admissible} as in \cite[Section 3]{Chen}. The above lemma shows that there is no obstruction on the level of characteristic monoids to lift $\uf$ to a stable log map. 
\end{remark}

For later use, we will identity the elements $l_i$ and $e_{ij}$ with their images $\chi_i(e_{ij})$ and $\Theta_{i}(l_i)$ in $\ocM$ when there is no danger of confusion.

\subsection{Lift to stable log maps}\label{ss:lift-to-log}\Yi{(1.32)}
\R{
Consider a pre-logarithmic comb $\uf: \uC \to D_{\lambda}$  as in Definition \ref{def:underlying-comb}.
} 
Let $C^{\sharp} \to S^{\sharp}$ be the log curve with the canonical log structure associated to the underlying pre-stable curve $(\uC, q_{\infty})$. We may fix a chart 
\[
\beta_{\sharp}: P_{\sharp} := \oplus_{i=1}^{m} N_i \to \cM_{S^{\sharp}}.
\]

Using the canonical map $\oplus_{i}\Theta_i: P_{\sharp} \to \ocM$ and the chart $\beta_{\sharp}$, we form a new log structure over $\uS := \uS^{\sharp}$ by
$
\cM := \ocM \oplus_{P_{\sharp}} \cM_{S^{\sharp}}.
$
Note that the inclusion 
$
\beta: \ocM \to \cM
$
defines a chart of $\cM$. Denote by $S = (\uS, \cM)$. The morphism $\cM_{S^{\sharp}} \to \cM$ defines a morphism of log schemes
$
S \to S^{\sharp}.
$
Denote by 
\[C := (\uC, \cM_{C}) = C^{\sharp} \times_{S^{\sharp}}S\] 
where the fiber product is taken in the category of log schemes. We thus obtain the log curve $\pi: C \to S$ over the underlying prestable curve $(\uC, q_{\infty})$. 

We now need to construct a log map $f: C/S \to X$ over $\uf$. This amounts to construct the map of log structures:
$
f^{\flat}: \uf^*\cM_{X} \to \cM_{C}.
$

Consider the standard basis $\{\delta_{j}\}$ of $P$ where $\delta_j$ corresponds to the component $D_j$. For convenience, we may identify $\delta_j$ with its image in $\uf^*\ocM_{X}$ when there is no danger of confusion. Denote by $\sigma$ the local coordinate near $q_{\infty}$, and \R{let} $\log \sigma$ be the corresponding image in $\cM_{C}$. For each node $r_i$, denote by $x_i$ and $y_i$ the two coordinates around $r_i$ on $\uC_0$ and $\PP_i$ respectively. Let $\log x_i$ and $\log y_i$ be the corresponding images in $\cM_{C}$. Choosing the coordinates properly, we may assume that 
\begin{equation}
\log x_i + \log y_i = l_i
\end{equation}
where $l_i$ is identified with its image  $\beta(l_i)$ in $\cM_C$. \R{Recall $\cM_X$ is given by the simple normal crossings boundary.} 
On the level of characteristic monoids, we have

\begin{lemma}\label{lem:lift-char}
There exists a unique morphism of sheaves of monoids
\[
\bar{f}^{\flat}: \uf^*\ocM_{X} \to \ocM_{C}
\]
determined by
\begin{enumerate}
 \item $\bar{f}^{\flat}(\delta_j) = e_{0j} + c_{\infty j}\cdot \log \sigma$ at the marking $q_{\infty}$;
 \item $\bar{f}^{\flat}(\delta_j) = e_{ij} + c_{ij} \cdot \log y_i$ at the node $r_{i}$.
\end{enumerate}
Here we identify $\log \sigma$ and $\log y_i$ with their corresponding images in $\ocM_{C}$.
\end{lemma}
\begin{proof}
It suffices to check the compatibility over the non-marked, smooth locus of $\uC$. One may then check that the compatibility is precisely the minimality of the monoid $\ocM$ given by (\ref{equ:min-monoid}).
\end{proof}

Denote by $\psi_{X}: \uf^*\cM_{X} \to \uf^*\ocM_{X}$ and $\psi_{C}: \cM_{C} \to \ocM_{C}$ the quotient morphisms. The inverse images $\cT_{j} := \psi^{-1}(\delta_j)$ and $\fT_j := \psi_{C}^{-1}(\bar{f}^{\flat}(\delta_j))$ \R{form $\cO^*_{C}$-torsors}. We observe that

\begin{lemma}\label{lem:torsor}
To define a log map $f: C/S \to X$ with $\bar{f}^{\flat}$ described as in Lemma \ref{lem:lift-char} is equivalent to have a set of  isomorphisms of torsors\[
\cT_{j} \to \fT_{j}, \ \ \ \mbox{for each }j \in \lambda.
\]
\end{lemma}
\R{
\begin{proof}\Yi{3.61: referee require a further justification -- using $\ocM_X$ is free. Qile:proof expanded!}
Denote by $<\cT_j> \subset \uf^*\cM_X$ the sub-log structure generated by $\cT_j$. Since $\uf^*\ocM_X$ is globally constant of the form $\NN^n$, we have the splitting
\[
\uf^*\cM_X = <\cT_1>\otimes_{\cO^*}\cdots\otimes_{\cO^*}<\cT_n>.
\]
Thus, to define a morphism $f^{\flat}: \uf^*\cM_X \to \cM_C$ is equivalent to define a morphism of log structures $<\cT_1> \to \cM_C$ for each $j \in \lambda$. 

Since we further require $f^{\flat}$ to be compatible with $\bar{f}^{\flat}$ defined in Lemma \ref{lem:lift-char}, the image of $\cT_1$ in $\cM_C$ factors through $\fT_j$ for each $j \in \lambda$. Furthermore, since $f^{\flat}$ is a morphism of monoids whose restriction $f^{\flat}|_{\cO^*}$ is the identity, $f$ is uniquely determined by the induced isomorphism of $\cO^*$-torsors $f^{\flat}|_{\cT_j}: \cT_j \to \fT_j$. This finishes the proof.
\end{proof}
}

\begin{proposition}\label{prop:lift}
There exists a stable log map $f: C/S \to X$ over the \R{pre-log comb} as in Definition \ref{def:underlying-comb} if and only if for each $j \in \lambda$, there is an isomorphism of line bundles
\begin{equation}\label{equ:bundle-iso}
N_{\uD_j}|_{\uC_0} \cong \cO_{\uC_0}(c_{\infty j}\cdot q_{\infty} - \sum_{i=1}^{m}c_{ij}\cdot r_i), \ \ \mbox{over }\uC_0.
\end{equation}
\end{proposition}
\begin{proof}
We notice that the restriction $(\cT_j)|_{\uC_0}$ is the torsor associated to $N^{\vee}_{D_j}|_{\uC_0}$, and $\fT_{j}|_{\uC_0}$ is the torsor associated to $\cO_{\uC_0}(- c_{\infty j}\cdot q_{\infty} + \sum_{i=1}^{m}c_{ij}\cdot r_i)$. Furthermore, the restriction $(\cT_j)|_{\PP_i}$ is the torsor associated to $\cO_{\PP_i}(-c_{ij})$, and $\fT_{j}|_{\PP_i}$ is the torsor associated to $\cO_{\PP_i}(-c_{ij} r_i)$. Since the curve $\uC$ is a comb with $\PP_i \cong \PP^1$, the statement follows from Lemma \ref{lem:torsor}.
\end{proof}

\begin{corollary}\label{cor:rat-lift}
\R{Let $\uf$ be a genus zero pre-log comb as in Definition \ref{def:underlying-comb}.} Then the log lift $f$ over $\uf$ exists.
\end{corollary}
\begin{proof}
In the genus zero case, the existence of isomorphisms (\ref{equ:bundle-iso}) follows from the degree consideration.
\end{proof}

%
%

\subsection{Deformation of Combs}\label{ss:comb-def}
%
%

We next study a more general situation where the combs are not necessarily contained in the center. \R{This generalization will not complicates the calculation, but is useful for studying the deformation of other types of combs \cite[Section 3]{rankone}.} \Yi{3.70, Qile: addressed here!}


%

\begin{hyp}\label{not:comb-smoothing}
Notations as in Section \ref{ss:underlying-comb}, consider the center $D_{\lambda}$ with $\uD_{\lambda} = \cap_{i \in \lambda} \uD_{i}$. Let $f: C/S \to X$ be a genus $g$ stable log map over a geometric point $\uS$ with contact markings 
\[q_{\infty 1},\cdots, q_{\infty k} \]
such that
\begin{enumerate1}
 \item $\uC = \cup_{i=0}^{m}\uC_i$ with smooth irreducible component $\uC_{i} \cong \PP^1$ for $i \neq 0$, and a smooth genus $g$ component $\uC_0$.

 \item For each $i \neq 0$, we have a node $r_i \in \uC$ joining $\uC_i$ and $\uC_0$. 
 
  \item \R{$\uf_0:= \uf|_{\uC_0}: \uC_{0} \to  \uD_{\lambda}$ is an immersion.}
 
 \item For $i \neq 0$, $\uf|_{\uC_i}$ is also an immersion.
 
   \item For $i\neq 0$, $\vec{c}_i := (c_{ij})_{j \in \lambda}$ \R{viewed as a $\kk$-vector is non-zero}\Yi{3.69, why??Qile:changde, it should be a $k$-vector, right?}, where $c_{ij} = f_*[C_i] \cdot \uD_j \in \ZZ$.
   
   \item \R{For each $l$, the image of the contact order of $q_{\infty l}$ (see \ref{equ:contact-order}) in $\ocM_{X, f(q_{\infty l})}^{\vee} \otimes_{\ZZ}\kk$ is non-zero.}

 
\end{enumerate1} 
\end{hyp}
In the above setting, \R{if we further require $\uf(\uC) \subset \uD_{\lambda}$, then we obtain the logarithmic comb in Definition \ref{def:underlying-comb}. 
The stable log map $f$ in Hypothesis \ref{not:comb-smoothing} may have log structures different then the one in Section \ref{ss:lift-to-log}. However, to analyze the deformation theory, only the existence of $f$ is needed. 
}



\begin{lemma}\label{lem:log-immerse}
\R{
Notations as in Hypothesis \ref{not:comb-smoothing},
} the morphism 
\[
\diff f: f^*\Omega_{X} \to \Omega_{C/S}.
\]
is surjective with locally free kernel, denoted by $N_{f}^{\vee}$. \Yi{3.72 Qile: explain referee's question by this lemma} 
%
%
\end{lemma}
\R{
\begin{proof}
It suffices to check the statement locally on the underlying curve $\uC$. Locally near points of $\uC$ where $\uf$ is a local immersion, the statement is identical to the classical situation without log structures. We next check the statement locally around a node $r_i \in \uC$. 

Denote by $x_i$ and $y_i$ the two local coordinates of the two components $\uC_i$ and $\uC_0$ around  $r_i$ respectively. Then the line bundle $\Omega_{C/S}$ locally around $r_i$ is generated by the section $\frac{\diff x_i}{x_i} = - \frac{\diff y_i}{y_i}$. 

On the target side, denote by $\{\delta_j \}_{j \in \lambda}$ the local sections of $\cM_{X}$ around $\uf(r_i)$ such that $\delta_j$ corresponds to the defining equation of $D_i$ near $\uf(r_i)$. Then the set of local sections $\{\frac{\diff \delta_j}{\delta_j}\}_{j\in \lambda}$ spans a rank $|\lambda|$ sub-bundle of $\Omega_{X}$ locally around $r_i$. By choosing coordinates $\{\delta_j\}$ carefully, and using Hypothesis \ref{not:comb-smoothing}, we may assume
\begin{equation}\label{equ:local-log-map}
f^{\flat}(\delta_j) = c_{ij} \log x_i
\end{equation}
where $\log x_i$ denotes the corresponding section in $\cM_{C}$ locally around $r_i$. We thus calculate that
\begin{equation}\label{equ:log-diff}
\diff f: (\frac{\diff \delta_j}{\delta_j})_{j \in \lambda} \mapsto \vec{c}_i \cdot \frac{\diff x_i}{x_i} = - \vec{c}_i \cdot \frac{\diff y_i}{y_i}.
\end{equation}
Since $\vec{c}_i$ is a non-zero $k$-vector, the desired statement around $r_i$ follows. 

The case of marked points follows from a similar calculation. 
\end{proof}
}


To calculate the deformation of log combs, we study the structure of the log cotangent bundle $\Omega_X$ along $D_{\lambda}$. 
\Yi{3.63 Qile: Simplified according to the referee}\Yi{I do not see the point here}

\begin{lemma}\label{lem:comp-cotangent} 
Notations as above, there is a natural exact sequence
\[
0 \to \Omega_{\uD_{\lambda}} \to \Omega_{X}|_{\uD_{\lambda}} \to \cO_{\uD_{\lambda}}^{\oplus |\lambda|} \to 0.
\]
\end{lemma}
\begin{proof}
For simplicity, we assume $D_{\lambda} \subset D_i$ for all $i$. We have an exact sequence
\begin{equation}\label{equ:restriction3}
0 \to N_{\uD_{\lambda}/\uX}^{\vee} \to \Omega_{\uX}|_{\uD_{\lambda}} \to \Omega_{\uD_{\lambda}} \to 0
\end{equation}
On the other hand, \R{we have the residue sequence}
\[
0 \to \Omega_{\uX} \to \Omega_{X} \to \sum_{i} \cO_{D_i} \to 0
\]

Tensoring with $\cO_{\uD_{\lambda}}$, we obtain the exact sequence:
\begin{equation}\label{equ:restriction}
0 \to \sum_{i} Tor^{\cO_{\uX}}_1(\cO_{\uD_{\lambda}},\cO_{\uD_{i}}) \to \Omega_{\uX}|_{\uD_{\lambda}} \to \Omega_{X}|_{\uD_{\lambda}} \to \cO_{\uD_{\lambda}}^{\oplus|\lambda|} \to 0
\end{equation}

To calculate the left hand side, we take the resolution
\[
0 \to \cO_{\uX}(-\uD_{i}) \to \cO_{\uX} \to \cO_{\uD_{i}} \to 0.
\]
Tensoring with $\cO_{\uD_{\lambda}}$, we have the exact sequence:
\begin{equation}\label{diag:compute-tor}
0 \to Tor^{\cO_{\uX}}_1(\cO_{\uD_{\lambda}},\cO_{\uD_{i}}) \to \cO_{\uD_{\lambda}}(-\uD_{i}) \to \cO_{\uD_{\lambda}} \stackrel{\cong}{\to}  \cO_{\uD_{\lambda}} \to 0
\end{equation}
\R{hence 
$Tor^{\cO_{\uX}}_1(\cO_{\uD_{\lambda}},\cO_{\uD_{i}}) \cong \cO_{\uD_{\lambda}}(-\uD_{i})$.} \R{The sequence (\ref{equ:restriction}) becomes}
\begin{equation}\label{equ:restriction2}
0 \to N_{\uD_{\lambda}/\uX}^{\vee} \to \Omega_{\uX}|_{\uD_{\lambda}} \to \Omega_{X}|_{\uD_{\lambda}} \to \cO_{\uD_{\lambda}}^{\oplus|\lambda|} \to 0
\end{equation}
Putting (\ref{equ:restriction2}) and (\ref{equ:restriction3}) together, we obtain the exact sequence as in the statement.
\end{proof}


\begin{lemma}\label{lem:attach-curve-key} 
\R{Assume Hypothesis \ref{not:comb-smoothing} holds.} We obtain the following commutative diagram with exact rows and columns:
\begin{equation}\label{diag:cotangent-compare}
\xymatrix{
 & 0 \ar[d] & 0 \ar[d]  & 0 \ar[d] & \\
0 \ar[r] & N^{\vee}_{\uf_{0}} \ar[r] \ar[d] & N_{f}^{\vee}|_{\uC_0} \ar[r] \ar[d] & V \ar[r] \ar[d] & 0 \\
0 \ar[r] & \uf^*_{0}\Omega_{\uD_{\lambda}} \ar[r] \ar[d] & f^*\Omega_{X}|_{\uC_0} \ar[r] \ar[d] & \cO^{\oplus |\lambda|} \ar[r] \ar[d]^{\phi} & 0\\
0 \ar[r] & \Omega_{\uC_{0}} \ar[r] \ar[d] & \Omega_{C}|_{\uC_0} \ar[r] \ar[d] & \sum_i \kk_{r_i}  \ar[r] \ar[d] & 0 \\
  & 0 & 0 & 0 &
}
\end{equation}
Furthermore, the restriction $\phi|_{r_i}$ is given by the $(1\times |\lambda|)$-matrix $(-\vec{c}_i)$. In particular, the vector bundle $V^\vee$ is the elementary transform of the trivial vector bundle along the direction $\vec{c}_i$, \R{see \cite[Lemma 3.15]{Starr}.}
\end{lemma}
\R{
\begin{proof}
Observe that the middle row is obtained by pulling back the exact sequence in Lemma \ref{lem:comp-cotangent}, and the middle column follows from Lemma \ref{lem:log-immerse}.

Now consider the bottom row. Since $\uC_0$ is a component of $\uC$, it suffices to consider the morphism 
\begin{equation}\label{equ:node-comp}
\Omega_{\uC_{0}} \to \Omega_{C}|_{\uC_0} 
\end{equation}
at each node $r_i$. Let $x_i$ and $y_i$ be the local coordinates around $r_i$ as in the proof of Lemma \ref{lem:log-immerse}. Then locally around $r_i$, the bundle $\Omega_{\uC_{0}}$ is generated by the local section $\diff y_i$, and $\Omega_{C}$ is generated by the local section $\frac{\diff y_i}{y_i}$. Thus, (\ref{equ:node-comp}) is the obvious injection with cokernel supported on each node $r_i$ as in the diagram. Indeed, $\kk_{r_i}$ is the torsion sheaf supported on $r_{i}$ with generator given by the image of $\frac{\diff y_i}{y_i}$. This proves the exactness of the bottom row.

Next, consider the commutativity of Diagram (\ref{diag:cotangent-compare}). By diagram chasing, it suffices to show the commutativity of the lower left conner which follows from $\uf_0 = \uf|_{\uC_0}$. Using the commutativity, we observe that
\[
N^{\vee}_{f_0} \to f^*\Omega_{X}|_{\uC_0} \to \Omega_C|_{\uC_0}
\]
is the zero morphism, hence it factors through $N^{\vee}_{f}|_{\uC_0}$. \Yi{3.75: Qile: explained the horizontal arrow on the upper left conner.}

Finally, we calculate the morphism $\phi$. Recall that $\kk_{r_i}$ is the torsion sheaf supported on $r_{i}$ with generator given by the image of $\frac{\diff y_i}{y_i}$. On the other hand, the morphism 
$
\Omega_{X}|_{D_{\lambda}} \to \cO^{\oplus |\lambda|}
$
is induced by taking the residue along each divisor $D_j$ with $j \in \lambda$. Thus, using the notations in the proof of Lemma \ref{lem:log-immerse}, for each $j \in \lambda$, the corresponding copy of $\cO$ in $\cO^{\oplus |\lambda|}$ has a generator given by the image of $\frac{\diff \delta_j}{\delta_j}$. Hence, locally around each $r_i$, the morphism $\phi$ is induced by (\ref{equ:log-diff}). In particular, $\phi|_{r_i}$ is given by $(-\vec{c}_i)$. 
\end{proof}
}

\subsection{Proof of Theorem \ref{thm:comb-smoothing}}\label{ss:comb-proof}
\R{The proof is divided into several steps.}

\vspace{3mm}

\R{
\noindent 
{\bf Step 1. Reduction to the simple normal crossings case.}

\begin{lemma}\label{lem:resolution-center}
Let $Y$ be a center of a log smooth variety $X$. Suppose $Y$ is proper, separably rationally connected, and fully free (resp. primitively free). For any proper, birational, log \'etale morphism $\pi: X' \to X$, there exists a center $Y'$ of $X'$ over $Y$ such that $Y'$ is fully free (resp. primitively free). 
\end{lemma}
\begin{proof}
Denote by $A$ the set of centers of $X'$ over $Y$. Consider a center $Y' \in A$. By Lemma \ref{lem:center-resolution}(3), the projection $\uY' \to \uY$ is indeed an isomorphism. Hence $\uY'$ is also separably rationally connected. 

By Lemma \ref{lem:center-DF}, we may take distinguished charts $\beta: P \to \ocM_{Y}$ and $\beta: Q_{Y'} \to \ocM_{Y'}$. The morphism $\pi: X' \to X$ induces an inclusion of cones $(Q_{Y'})^{\vee} \hookrightarrow P^{\vee}$ such that $(Q_{Y'}^{gp})^{\vee} = (P^{gp})^{\vee}$. We may view $(Q_{Y'})^{\vee} $ as a sub-cone of $P^{\vee}$. Since $\pi$ is proper, log \'etale, and birational,  we have 
\[
P^{\vee} = \cup_{Y' \in A} (Q_{Y'})^{\vee}.
\]
Thus, by Lemma \ref{lem:curve-in-center}, every log admissible curve in $Y$ is also log admissible in one of the center in $A$. Since the set $A$ is finite, there exists some $Y' \in A$ which has a log admissible free rational curve $\uF \to \uY'$ such that $\co_{Y'}(\uF)$ is contained in the interior of $(Q_{Y'})^{\vee}$. We next show that such $Y'$ is fully free.

Consider a set of log admissible free rational curves $B = \{\uF_1, \cdots, \uF_k\}$ on $\uY$ whose total contact orders span $P^{\vee}\otimes_{\ZZ}\kk$. We may choose a sufficiently large $m \in \NN$ such that $B' = \{\uF'_1 = m\cdot \uF + \uF_1, \cdots, \uF'_k = m \cdot \uF + \uF_k\}$ are all log admissible. The smoothing techniques in \cite[II.7]{kollar} imply that the curve classes in $B'$ can be represented by free rational curves in $\uY'$. Thus $Y'$ is fully free. 

The primitivity of $Y'$ follows from the same argument.
\end{proof}
}

By \cite[Theorem 5.10]{Ni}, there exists a birational, log \'etale resolution $\pi: X' \to X$ such that $X'$ has simple normal crossings boundary.  Using Lemma \ref{lem:resolution-center}, we may assume $X$ has simple normal crossings boundary. \Qile{3.77: add another prop to solve this problem.}

\vspace{2mm}
\R{
\noindent
{\bf Step 2. Construct the log comb.}

By assumption, there exists free rational curves $\uR_1, \cdots, \uR_l$ on $\uY$ such that their total contact orders lie in $P^\vee$, and span the {$\kk$-}vector space $(P^{gp})^{\vee}\otimes_{\ZZ}\kk$. Then there exists an open subset $\uV\subset \uY$ such that every point in $\uV$ is contained in the image of a free deformation of $R_i$ for all $i$.
Now we construct a pre-log comb:
\begin{enumerate1}
	\item (handle) a very free rational curve $\uf_0 : \uC_0 \to \uY$ whose general point lies in $V$;
	\item (teeth) for $i = 1, \cdots, m$, a free rational curve $\uf_i:C_i:=\PP^1 \to \uY$ in $Y$ which is a deformation of some $R_j$. 
	\end{enumerate1} 
For $i \neq 0$, glue $\uf_{i}$ to $\uf_0$ along points $p_i \in \uC_i$ and $q_i \in \uC_0$. This yields a pre-log comb $\uf: \uC \to \uY$ as in Definition \ref{def:underlying-comb}. We may pick a general smooth point $q_{\infty} \in \uC_0$. By Corollary \ref{cor:rat-lift}, we can lift $\uf$ to a stable log map $f: C/S \to X$ with a unique marking $q_{\infty}$. A figure of the underlying curve $\uC$ is depicted below: \Qile{1.42:figure added.}
}

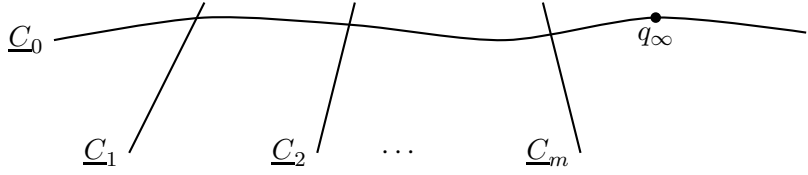
\vspace{2mm}

\begin{center}

\begin{tikzpicture}
\draw [black, thick, xshift=4cm] plot [smooth, tension=0.5] coordinates { (0,0) (2,.3) (4,.2) (6,0) (8,.3) (10, .1)};
\node [left, xshift=4cm] at (0,0) {$\uC_0$};

\draw [black, thick, xshift=4cm] (2,.5) -- (1,-1.5);
\node [left, xshift=4cm]  at (1,-1.5) {$\uC_1$};

\draw [black, thick, xshift=4cm] (4,.5) -- (3.5,-1.5);
\node [left, xshift=4cm]  at (3.5,-1.5) {$\uC_2$};

\node [left, xshift=4cm] at (5,-1.5) {$\cdots$};

\draw [black, thick, xshift=4cm] (6.5,.5) -- (7,-1.5);
\node [left, xshift=4cm]  at (7,-1.5) {$\uC_m$};

\node [dot=black, xshift=4cm] at (8,.3) {};

\node [below, xshift=4cm] at (8,.3) {$q_{\infty}$};

\end{tikzpicture}

\end{center}




\vspace{1mm}

\R{
Consider the log smooth variety $X \times \PP^1$ where we equip $\PP^1$ with the trivial log structure. The fully free center $Y$ gives a fully free center $Y\times \PP^1$ of $X\times \PP^1$. We note that any free (resp. very free) rational curve $u:\P^1\to \uY$ gives a free (resp. very free) immersed rational curve on $\uY\times \P^1$ by taking its graph curve, and  any very free $\AA^1$-curves on $X\times \PP^1$ project to a very free $\AA^1$-curve on $X$. Thus, replacing $X$ by $X\times \PP^1$, we may assume that $\uf$ is an immersion on each irreducible component of $\uC$.

\vspace{2mm}

\noindent
{\bf Step 3. Analyzing the positivity.}
We are now in the situation of Hypothesis \ref{not:comb-smoothing} with $D_{\lambda} = Y$. We next analyze the positivity of $N_f = (N_{f}^{\vee})^{\vee}$ as in Lemma \ref{lem:log-immerse}. Pick two general points $t_1$ and $t_2$ on $\uC_0$, we have 
\begin{equation}\label{eq:s}
0 \to \cup_{i\neq 0} N_f|_{C_i}(-p_i)\to N_f(-t_1-t_2) \to N_f|_{C_0}(-t_1-t_2)\to 0.
\end{equation}

For each $i\neq 0$, since $C_i$ is a free rational curve in the center, using the left vertical exact sequence in (\ref{equ:node-comp}), we have \Yi{this may require a lemma.Qile: fixed!} 
\[
H^1(N_f|_{C_i}(-p_i))=H^1(N_{C_i}(-p_i))=0.
\]

Since $m$ can be sufficiently large, and the total contact orders of $\uR_1,\cdots,\uR_l$ span $(P^{gp})^{\vee}\otimes_{\ZZ}\kk$, by Lemma \ref{lem:attach-curve-key} and \cite[Lemma 3.15]{Starr}, we conclude that 
$H^1(N_{f}(-t_1-t_2)|_{\uC_0})=0.$
By (\ref{eq:s}), we have $H^1(N_f(-t_1-t_2))=0$. By \cite[Theorem 5.9]{Logcot}, \Yi{(1.37): Qile: precise reference added} $f$ is unobstructed, and a general deformation of $f$ produces an $\AA^1$-curve with positive log normal bundle, thus a very free $\A^1$-curve. \qed
}


\begin{corollary}\label{cor:lift-center-rat}
Let $X$ be a log smooth variety with a proper, Deligne-Faltings type center $Y \subset X$. Then a rational curve $\uf: \uC = \PP^1 \to \uY$ can be \R{lifted} to a \R{stable log map with at most one marking} $f: C/S \to X$ if and only if the curve $\uf$ is log admissible. \R{If further assume $\uf$ is free on $\uY$, then a general deformation of $f$ is a free log rational curve. In particular, $X$ is separably log uniruled.}
\Yi{(3.12)} \Qile{Referee 3 complains that hard to appreciate this result if $\co_{\beta}$ is not defined.}\Yi{wrong place}
\end{corollary}
\begin{proof}
\R{
By a similar argument as in Lemma \ref{lem:resolution-center}, it suffices to consider $X$ has simple normal crossings boundary. The existence of the log lift follows from Corollary \ref{cor:rat-lift}. By Lemma \ref{lem:comp-cotangent} and  \cite[Theorem 5.9]{Logcot}, $f$ is unobstructed whose general deformation yields a free log rational curve.
}
\end{proof}



\section{Wonderful comapactifications}\label{sec:canonical}

\subsection{Basics on spherical varieties}\label{not:spherical}

Let $G$ be a linearly reductive $\kk$-group. Let $T$ be a maximal torus of $G$ and $B$ be a Borel subgroup containing $T$.  For any subgroup $H \subset G$, denote by $H^{u}$ the unipotent radical of $H$. We write $\fX^*(H)$ and $\fX_{*}(H)$ for the character and cocharacter groups of $H$ respectively.

\begin{definition}\label{def:spherical}
Let $\uX$ be a $G$-variety. The variety $\uX$ is a \emph{separably spherical variety} if it is normal, and contains a dense open separable $B$-orbit. A subgroup $H \subset G$ is separably spherical if $G/H$ is so.
\end{definition}


For reader's convenience, we collect a list of standard terminologies and results that are known to experts, \R{and will be used in this section}. Those can be found in standard references, for example \cite{Knop}. 

\begin{term}{\ }\label{term}
Let $G/H$ be a spherical homogeneous space. We fix a point $o \in G/H$ whose $B$-orbit is dense. We use $\uX$ to denote a spherical variety which contains $G/H$ as the open $G$-orbit. \Qile{I move this paragraph here together with the terminologies!}
\begin{enumerate1}

\item \R{Let $\kk(\uX)$ be the field of rational functions on $\uX$. }Denote by $\kk(\uX)_o^{(B)}$ the set of $B$-eigenfunctions ($B$-semi-invariant functions) on $\uX$ given by
	\[\{f\in \kk(\uX)\backslash\{0\}|f(o)=1,bf=\chi(b)f, \forall b\in B,\text{where } \chi\in \fX^*(B)\}.\]
	\R{Since we require that $f(o)=1$, such a $B$-eigenfunction $f$ is uniquely determined by its character $\chi$.}
	
	
\item Let $\Lambda := \Lambda(\uX)$ be the set of weights of $\kk(\uX)_o^{(B)}$. It is a finitely generated free abelian group. Its rank is called the rank of $G/H$. Indeed we have
	\[ \kk(G/H)_o^{(B)}\cong \Lambda.\]
	We use $f_\lambda$ to denote the $B$-eigenfunction determined by $\lambda\in \Lambda$.\Yi{3.84,85 fixed needs to explain in the letter we require a fixed  value on o.}
	
\item Define the valuation space $N(\uX) := Hom_\Z(\Lambda,\Q)$ with the integral structure $N(\uX)_\Z=\Lambda^\vee$.
	
\item Let $\Delta(\uX)$ be the set of $B$-stable but not $G$-stable prime divisors on $\uX$, called the set of \emph{colors}. \R{The set $\Delta(\uX)$ only depends on $G/H$, see \cite[2]{Knop}. }\Yi{3.86 fixed}
	
\item Any discrete valuation $\nu:\kk(\uX)^*\to \Q$ gives an element $\rho_\nu$ in $N(\uX)$ by restriction to $\kk(\uX)_o^{(B)} $.
	
\item Let $D_G(\uX)$ be the set of $G$-invariant valuations on $\uX$. The valuation map $D_G(\uX)\to N(\uX)$ is injective. Denote by $\cV(\uX)$ the $\QQ$-cone generated by the image of $D_G(\uX)$ in $N(\uX)$. This is a convex cone, called the {\em valuation cone}.
	
	
\item Since $\Lambda(\uX), N(\uX), \Delta(\uX), \cV(\uX)$ defined above only depend on the dense $G$-orbit $G/H$, we simply omit $\uX$ and use $\Lambda, N, \Delta$, and $\cV$.
	
\item The valuations given by $\Delta$ induce a map $\rho|_\Delta: \Z^\Delta\to N_{\ZZ}.$
	
\item A spherical variety $\uX$ is called \emph{toroidal} if no $\uD\in \Delta$ contains a $G$-orbit in its closure. It is called \emph{simple} if it has a unique closed $G$-orbit. We say that $\uX$ is a \emph{compactification} of $G/H$ if it is proper.
	
\item A spherical subgroup $H \subset G$ is called {\em sober} if $N_G(H)/H$ is finite. 
	
\item When $H$ is sober and separably spherical, the valuation cone $\cV$ associated to $G/H$ is a strictly convex full dimensional cone in $N$, see \cite[Corollary 5.3, Theorem 6.1]{Knop}. In this case, any $G$-equivariant toroidal compactification of $G/H$ is uniquely determined by a fan $\Sigma$ supported on $\cV$. The fan $\Sigma$ and \R{the cone} $\cV$ are called the {\em colored fan} and {\em colored cone} respectively. 
	
\item For a sober and separably spherical subgroup $H$, the compactification $\u{\X}$ of $G/H$ associated to the colored fan $(\cV,\emptyset)$ is called the {\em wonderful compactification} of $G/H$. In this case $\u{\X}$ is both toroidal and simple with a unique closed orbit $\uY \subset \u{\X}$. 

\item \R{For any spherical $G$-variety $\uX$ and any $G$-orbit $\uY\subset \uX$, we define the open set:
	$\uX_{\uY,B}=\uX\backslash \cup \uD,$
	where the union is taken over all $B$-stable prime divisors that do not contain $\uY$.} \Yi{(1.49) fixed}
\end{enumerate1}
\end{term}

\subsection{Log structures on toroidal embeddings}\label{ss:toroidal}
The goal of this section is to define a natural $G$-equivariant log structure on toroidal spherical varieties, and to find a good criterion for a spherical variety to be log homogeneous in the following sense.

\R{
\begin{definition}\label{def:G-log} \Yi{3.89 not sure how to fix take a look; Qile: is this OK?}
Let $H$ be an algebraic group. $H$ as a log scheme with the trivial log structure can be viewed as a group object in the category of log schemes, see \cite[Definition 2.11]{Vistoli}. An $H$-action on a log scheme $Y$ is the action of the group object $H$ on $Y$ in the category of log schemes, see \cite[Definition 2.15]{Vistoli}. A log scheme $Y$ equipped with an $H$-action is called a \emph{log $H$-scheme}.
\end{definition}

\begin{remark}
To be more precise, an $H$-action on a log scheme $Y$ is a strict morphism of log schemes $H \times Y \to Y$ whose pull-back to $H\times H \times Y$ satisfies the usual cocycle condition, see \cite[Proposition 2.16]{Vistoli}. When $Y$ is a log smooth variety, the structure morphism $\cM_Y \to \cO_{\uY}$ is an inclusion. Thus, an $H$-action on $Y$ is uniquely determined by the corresponding $H$-action on $\uY$ with the induced action on the $\cM_Y$ as the subsheaf of $\cO_{\uY}$. In what follows, we will mainly consider $H$-actions on a log smooth variety.
\end{remark}
}


\begin{definition}\label{def:log-homogeneous}
A log smooth log $H$-variety $X$ is called \emph{log homogeneous} if the morphism of sheaves
$\mathfrak{h}\otimes \O_{\uX}\to T_{X}$ is surjective, where $\mathfrak{h}$ is the Lie algebra of $H$, \R{and $T_{X}$ is the log tangent bundle of the log scheme $X$}.
\end{definition}

The main result of this subsection is the following:

\begin{proposition}\label{prop:log-homogeneous}
Let $\uX$ \R{be} a toroidal compactification of a separably spherical homogeneous space $G/H$. Assume all closed $G$-orbits of $\uX$ are separable. Then
\begin{enumerate}
 \item the pair $(\uX, \uX\setminus G/H)$ is toroidal, which yields a log smooth log $G$-variety $X$;
 \item the log smooth log $G$-variety $X$ is log homogeneous. 
\end{enumerate}
\end{proposition}

We first recall the following local structure result of spherical varieties. 

\begin{thm}[Local structure theorem] \label{thm:local-s} 
Let $\uX$ be a spherical \R{$G$-variety}. Let $\uY\subset \uX$ be a $G$-orbit containing \R{a geometric point} $x$ such that $Bx$ is open in $\uY$. Define the parabolic subgroup $Q$ to be the stabilizer of $\uX_{\uY,B}$ and choose a Levi subgroup $L$ of $Q$. Then: 
\begin{enumerate1}
\item $\uX_{\uY,B}$ is affine $B$-stable;
\item \R{Let $T$ be the maximal torus in $B$.} There exists an affine $T$-stable closed subvariety $M$ of $\uX_{\uY,B}$ such that the \R{group action}\Yi{1.51 fixed but not very sure} morphism $$\mu:Q^u\times \uM\to \uX_{\uY,B}$$ is finite surjective and $\mu^{-1}(x)=\{(e,x)\}$, where $e$ is the identity;\Yi{3.90 fixed}
\item If we further assume that $\uY$ is a separable $G$-orbit, the action morphism $\mu$ is an isomorphism;
\item If $\uY$ is $G$-separable and $\uX$ is toroidal, the variety $\uM$ is an affine toric embedding of a torus $A$ given by the quotient $L/L_0$, where $L_0 \subset L$ is a subgroup containing the derived group of $L$. Furthermore, we have $\Lambda\cong\fX^*(A)$ via the restriction of $B$-eigenfunctions on $M$. The toric variety $M$ is determined by the colored cone associated to $\uY$;
\item With the same assumptions in (4), every $G$-orbit of $\uX_{\uY,B}$ is of the form $\overline{Q^u\cdot \uM'}$, where $\uM'$ is an $A$-orbit. In particular, there is a bijection between $G$-orbits in $\uX$ and $A$-orbits in $\uM$. 
\end{enumerate1}

\end{thm}

\begin{proof} 
Statement (1) and (2) are proved in \cite[Theorem 1.2]{Knop93}. 

For (3), since we can $G$-equivariantly embed $\uX$ into $\P(V)$, where $V$ is a simple $G$-module, it suffices to prove the case when $\uX=\P(V)$ (not necessary spherical) and $\uX_{\uY,B}$ is the complementary in $\P(V)$ of the unique $B$-stable hyperplane. This is proved in \cite[Proposition 1.3]{Huruguen}. 

The statement (5) and the first assertion of (4) are proved in \cite[Theorem 1.4]{Huruguen} (requiring $\uM$ to be smooth is not necessary). 

By (3) and (5), it is easy to see that there is a bijection between $B$-eigenfunctions on $\uX$ and $A$-eigenfunctions on $M$ via restriction, which preserve the order given by inclusion of orbits. This proves (4).
\end{proof} 

\begin{lemma}\label{lem:closed-orbit}
A spherical variety $\uX$ is covered by $G$-translations of $\uX_{\uY,B}$ for any $G$-orbit $\uY$. Furthermore, when $\uX$ is proper, one may choose $\uY$ to be any closed \R{$G$-orbit}.
\end{lemma}
\begin{proof}
The first statement follows from \cite[Section 2.1]{Knop}. For each $\uY$, the simple embedding $G\uX_{\uY,B}$ of $G/H$ gives a strictly colored cone $(\cC,\cF)$ \cite[Theorem 3.3]{Knop}. When $\uX$ is proper, it belongs to a strictly colored cone of maximal dimension given by $G\uX_{\uY',B}$ for a closed $G$-orbit $Y'$ with $G\uX_{\uY,B} \subset G\uX_{\uY',B}$. The second statement then follows.
\end{proof}

\begin{proof}[Proof of Proposition \ref{prop:log-homogeneous}]
Statement (1) follows from Theorem \ref{thm:local-s} and Lemma \ref{lem:closed-orbit}. Now consider the second statement. By further subdividing the colored fan of $\uX$, we may take a $G$-equivariant log \'etale birational morphism $f:Z\to X$ such that $\uZ$ is smooth. Since the log structure is Zariski and $f$ is $G$-equivariant, any closed orbit of $Z$ maps isomorphically onto a closed orbit on $X$. Thus the $G$-action on the closed orbits of $Z$ is separable. By \cite[Thm. 1.8]{Huruguen}, $Z$ is log homogeneous. Since $f^*T_X = T_Z$ and toric singularities are rational, the projection formula implies that $f_*T_Z=T_{X}$. Hence $X$ is log homogeneous as well. 
\end{proof}

\subsection{Global charts of wonderful compactifications}\label{ss:wonderful}

\R{For the rest of this section, we deal with wonderful compactifications. By \cite[Lemma 1]{Brion-onecycle}, we may assume that $G$ is semisimple. } \Yi{3.87, 1.47 fixed}
\begin{notation}
Let $G$ be a semi-simple group, and $H \subset G$ be a sober and separable spherical subgroup. Denote by $\u{\X}$ the wonderful compactification of $G/H$ as in Terminology \ref{term} (10),(11), and (12). Since $\u{\X}$ is simple, denote by $\uY \subset \u{\X}$ its unique closed orbit. 

If $\uY$ is $G$-separable, we obtain a log smooth variety $\X=(\u{\X},\M)$ with the $G$-action by Proposition \ref{prop:log-homogeneous}. Denote by $\uX_1,\cdots,\uX_t$ the irreducible boundary components of $\X$, and $\Delta=\{\uD_1,\cdots,\uD_s\}$  the set of colors.  Let $\u{\X}_\circ := \u{\X}\backslash \Delta$ be the dense $B$-stable open subset, and $\X_\circ$ the corresponding log variety. The stabilizer group of $\u{\X}_\circ$ is a parabolic subgroup $Q$. Denote by $P := \cV^\vee\cap\Lambda$.
\end{notation}

\begin{lemma}\label{lem:chart}
Assume $\uY$ is $G$-separable. With the same notation as in Theorem \ref{thm:local-s}, for each $B$-eigenfunction $f_\lambda$ with $\lambda\in P$, we have the following:\begin{enumerate}
\item $f$ is regular on $\u{\X}_\circ$.
\item $p_2^*(f|_M)=f$, where $p_2:\u{\X}_\circ\cong Q^u\times M\to M$. 
\end{enumerate}
In particular, we obtain a chart
\begin{equation}\label{equ:wonder-chart}
P\to \M_{\circ}:=\M|_{\u{\X}_\circ}, \ \ \ \lambda \mapsto f_{\lambda}.
\end{equation}
\end{lemma}
\begin{proof} 
(1) is proved in \cite[Theorem 2.5]{Knop}. For (2), it suffices to show that each $B$-eigenfunction is constant on $Q^u\times \{m\}$ for every $m\in M$. This follows from the fact that $Q^u$ is a unipotent subgroup of $B$ which only has the trivial character.
\end{proof} 

\begin{proposition}\label{prop:wonderful-DF}
Assume $\uY$ is $G$-separable. There exists a global chart 
$\beta:  P \to \ocM$
induced by (\ref{equ:wonder-chart}), whose restriction to the unique center $Y \subset \X$ induces the distinguished chart  as in Definition (\ref{def:DF-center}).
\end{proposition}
\begin{proof} 
We first notice that the sheaf of \R{monoids} $\ocM|_{\uY}$ is a globally constant sheaf of \R{monoids} in $P$. Thus, it suffices to construct a global chart  $\underline{P} \to \overline{\M}:=\M/\O^*$. 

By Lemma \ref{lem:closed-orbit}, $\u{\X}$ is covered by $\{g\u{\X}_\circ\}_{\{g\in G\}}$. For each open subset $g\u{\X}_\circ$, we construct the morphism $\u{P}\to \o{\M}|_{g\u{\X}_\circ}$ sending $\lambda$ to $gf_\lambda$ similar to (\ref{equ:wonder-chart}). To check \R{that} this local morphism gives a global morphism from $\uP$ to $\overline{\M}$, it suffices to show that $gf_\lambda$ and $f_\lambda$ differ by an invertible function on the common intersection. Consider the rational function $gf_\lambda/f_\lambda$ on $\u{\X}_\circ\cap g\u{\X}_\circ$. By Lemma \ref{lem:chart}, it has nontrivial zeros or poles only at the boundary divisors $\uX_i\cap\u{\X}_\circ\cap g\u{\X}_\circ$'s. Since $\u{\X}_\circ$ is normal, and each prime boundary divisor corresponds to a $G$-invariant valuation, $gf_\lambda/f_\lambda$ is invertible on $\u{\X}_\circ\cap g\u{\X}_\circ$. 
\end{proof} 

The above global chart $\beta$ is natural in the following sense:

\begin{proposition}\label{prop:dual-pic-map}
Assume $\uY$ is $G$-separable, and consider the morphism as in  (\ref{equ:DF-line-bundle})\Yi{need to change; Qile: changed!}
\[L:\Lambda\cong P^{gp}\to \pic \u{\X}\]
induced by the global chart $\beta$. Recall that in this situation, the group $\pic\u{\X}$  is freely generated by the colors by \cite[Theorem 1]{Brion-onecycle}. Then we have
\begin{enumerate}
 \item $L(\lambda)=\O_{\u{\X}}(\sum_{i=1}^t\nu_{\uX_i}(f_\lambda)\uX_i)$;
 \item the map $-L^\vee$ is the valuation morphism on colors $\rho|_\Delta:\Z^\Delta\to N$ as in Terminology \ref{term} (8).
\end{enumerate}
\end{proposition}
\begin{proof}
The first statement follows from the proof of Proposition \ref{prop:wonderful-DF}. 
Since each eigenfunction $f_\lambda$ has only poles or zeros along $\uX_i$'s or $\uD_i$'s, we have:\[\text{div} f_\lambda=\sum_{i=1}^t \nu_{\uX_i}(f_\lambda)\uX_i+\sum_{j=1}^s \nu_{\uD_j}(f_\lambda)\uD_j.\]
Thus we have the second statement.
\end{proof}

\subsection{$\A^1$-curves on wonderful compactifications}\label{ss:A1-on-wonderful}
We next introduce two \R{technical} assumptions, which greatly simplifies the structure of the set of $\AA^1$-curve classes on $\X$. Later we will verify those two assumptions for a large class of interesting cases.

\begin{hyp}\label{hyp:knop}
There exists $B$-invariant irreducible rational curves 
\[\uB_1,\cdots,\uB_s\] 
on $\uY$ such that $\uB_i \cap \uD_j=\delta_{ij}$ and $\NE(\u{\X})=\N \langle \uB_1,\cdots,\uB_s \rangle$.
\end{hyp}

\R{
By Proposition \ref{prop:wonderful-DF}, the following is compatible with Definition \ref{def:log-one-cycle}.
\begin{definition}\label{def:log-one-cycle-X}
	The \emph{total contact order} of a curve class $F \in N_1(\u\X)$ is the element $\co(F)\in (P^{gp})^\vee$ defined by 
	\begin{equation}
		\co(F)(\delta) = c_1(L_{\beta}(\delta)) \cdot F, \ \ \ \mbox{for any }\delta \in P^{gp}.
	\end{equation}
 A curve class $F$ is \emph{log admissible} if $\co(F) \in P^\vee$. Denote by $\NE(\X) \subset \NE(\u\X)$ the semi-group of log admissible effective one-cycles on $\uY$.\Qile{Definition move to here from Section 3!}
\end{definition}
}


\begin{lemma} \label{lem:back-to-log-class}
Assume $\uY$ is $G$-separable, and \R{that} Hypothesis \ref{hyp:knop} holds. For $i=1,\cdots,s$, we have
\begin{enumerate}
 \item $\deg L(\lambda)|_{\uB_i}=-\nu_{\uD_i}(f_\lambda)$,  and hence $\co(\uB_i)=-\nu_{\uD_i}=-\rho|_\Delta(\uD_i)$ with $\co$ introduced in Definition \ref{def:log-one-cycle};

 \item $\NE(X)=\{F=\sum_{i=1}^s a_i\uB_i \ |\ a_i\in \ZZ_{\geq 0}, -\sum_{i=1}^s a_i\rho|_\Delta(\uD_i)\in \cV\}$;\Qile{Was $\NE(\X)$ instead of $\NE(Y)$.}

\end{enumerate}
\end{lemma}
\begin{proof} 
This follows from Proposition \ref{prop:dual-pic-map}.
\end{proof}

We further consider the following:
\begin{hyp}\label{hyp:cone}
The opposite of the valuation cone $-\cV$ is contained in the $\QQ$-cone generated by $\rho(\uD_i)$, for $i=1,\cdots,s$. 
\end{hyp}

\begin{lemma}\label{lem:torsion-free}
Assume Hypothesis \ref{hyp:cone} holds. Then the morphism $L:\Lambda\to \pic(\X)$ is injective. Furthermore, the following are equivalent:
\begin{enumerate}
\item $-\rho|_\Delta: \Z^\Delta\to \Lambda^{\vee}$ is surjective;
\item The cokernel of $L:\Lambda\to \pic(\u{\X})\cong\Z^\Delta$ is tosion-free.
\end{enumerate}
\end{lemma}
\begin{proof}
Since $\cV$ is a strictly convex cone in $N$, see Terminology \ref{term} (11), the injectivity follows from Proposition \ref{prop:dual-pic-map}, and the observation  that the cokernel of $-\rho|_\Delta$ is finite. The rest of the statement is a direct consequence of the injectivity. 
\end{proof}

We summarize the discussion as follows:
\begin{thm}\label{thm:main-wonderful}
Let $\X$ be the log smooth variety associated to the wonderful compactification of a sober separably spherical homogeneous space $G/H$, \R{containing a} unique $G$-separable closed orbit, see Proposition \ref{prop:log-homogeneous}. Assume Hypothesis \ref{hyp:knop} and \ref{hyp:cone} hold. Then we have:
\begin{enumerate}
\item Any log-admissible effective curve class on $\u{\X}$ can be represented by a free \R{log rational} curve on $\X$;\Qile{was a free $\A^1$-curve on $G/H$, changde!} 


\item  If $\ch \kk \nmid [\Lambda^{\vee} : \rho|_\Delta(\ZZ^{\Delta})]$, then the unique $G$-closed orbit is a fully free center of $\X$. In particular, $\X$ is separably $\A^1$-connected. 
\item Furthermore, the center of $\X$ is \R{primitively free} if and only if the cokernel of $L:\Lambda\to \pic(\X)\cong\Z^\Delta$ is tosion-free.
\end{enumerate}
\end{thm}

\begin{proof} 
Hypothesis \ref{hyp:cone} implies that $\NE(Y)$ is not empty. For any element $F=\sum_{i=1}^s a_i\uB_i\in \NE(\X)$, there exists a rational curve $\u{f}:\P^1 \to \uY$ whose curve class is $F$ since $\uY$ is a homogeneous space. By Corollary \ref{cor:lift-center-rat}, there exists a log map $f$ lifting $\uf$ as \R{log rational} curve on $\X$. By Proposition \ref{prop:log-homogeneous}, $f$ is unobstructed in the moduli space of $\A^1$-curves,  $f$ can be deformed to a \R{log rational} curve on $\X$. This implies (1). 

\R{
For the second statement, we first choose a basis $F_i$ of $\rho|_\Delta(\ZZ^{\Delta})$ with $F_i\in NE(\X)$. Since each $F_i$ can be represented by a free rational curve on $\uY$, the assumption on characteristic of the base field $\kk$ implies that $Y$ is a fully free center. By Theorem \ref{thm:comb-smoothing}, $\X$ is $\A^1$-connected. Finally (3) is a direct consequence of Lemma \ref{lem:torsion-free} and Definition \ref{def:full-free}.}
\end{proof}

\subsection{The characteristic zero case}\label{ss:char-zero-wonderful}

In this section, we adopt the \R{notation of} Section \ref{ss:wonderful}, and assume $\ch\kk=0$. The goal is to give a proof of Theorem \ref{thm:char0} by verifying the assumptions in Theorem \ref{thm:main-wonderful}. We may assume $G$ is simply connected for what follows. 



To verify Hypothesis \ref{hyp:knop}, we recall the idea of Luna on spherical closures.

\begin{definition}\cite[Section 6.1]{Luna}\label{def:spherically-closed} 
Let $K$ be a spherical subgroup of $G$. The automorphism group $N_G(K)/K$ naturally acts on $\Delta$. The \emph{spherical closure} $\overline{K}$ of $K$ is the kernel of the action of $N_{G}(K)$ on $\Delta$. We say that a subgroup $L$ is \emph{spherically closed} if $\overline{L}=L$.
\end{definition}

Recall the following facts from \cite[Remark 30.1]{Timashev}:
\begin{lemma}\label{lem:spherical-closure}
Let $\o{H}$ be the spherical closure of $H$, and $\X'$ be the log variety associated to the wonderful compactification of $G/\o{H}$. There is a $G$-equivariant morphism of log varieties $\pi:\X\to \X'$ such that: 
\begin{enumerate}
\item ${\u{\X}'}$ is smooth, see \cite[Corollary 7.6]{Knop96};
\item Let $\Lambda'$ and $\cV'$ be the weight lattice of $\X'$ and the valuation cone of $\X'$ respectively. We have that $\cV=\cV'$ and $\Lambda'\subset \Lambda$. In particular, $\pi$ is a finite log \'etale morphism (or equivalently Kummer \'etale), and $\pi|_{\uY}$ is an isomorphism onto the closed $G$-orbit $\uY' \subset \u{\X}'$;
\item the set of colors ${\Delta'}=\{{\uD'_1},\cdots,{\uD'_s}\}$ on ${\X'}$ is identified with $\Delta$ by pullback, i.e., $\pi^{-1}({\uD'_i})=\uD_i$.\qed
\end{enumerate}
\end{lemma}



We briefly recall the types of colors on wonderful compactifications \cite{Luna}. Let $\Sigma$ be the finite set of spherical roots of $\uX$, lying in the character lattice of $T$. See \cite[Proposition 6.4]{Luna}. We say a simple root $\alpha$ \emph{moves} a color $\uD$ if $P_\alpha \uD\neq \uD$ where $P_\alpha$ is the minimal parabolic group containing $B$ and associated to $\alpha$. Each color $\uD$ is moved by a unique simple root $\alpha_{\uD}$. Let $\Delta(\alpha)$ be the set of colors moved by $\alpha$.  
\begin{definition}\label{def:b}
We say that the color $\uD$ is 
\begin{enumerate}
\item of type (a) if $\Delta(\alpha_{\uD})$ contains two colors;
\item of type (a') if $\Delta(\alpha_{\uD})=\{\uD\}$ and $2\alpha\in \Sigma$;
\item of type (b) if $\Delta(\alpha_{\uD})=\{\uD\}$ and no multiple of $\alpha$ is in $\Sigma$.
\end{enumerate} 

\end{definition}
By \cite[Section 1.4]{Luna}, each color belongs to a unique type as above.

\begin{proposition}\label{prop:one-cycle}
Assume $\ch\kk=0$. Then Hypothesis \ref{hyp:knop} holds when all colors of $\u{\X}$ are of type (b). 
\end{proposition}
\begin{proof} This follows from \cite[Theorem 2.4 and Lemma 2.12]{irr-rat}.
\end{proof}


\begin{proposition}\label{prop:cone}
Assume $\ch\kk=0$. Then Hypothesis \ref{hyp:cone} holds.
\end{proposition}
\begin{proof} 
Consider the $G$-equivariant morphism $\pi:\X\to{\X'}$ as in Lemma \ref{lem:spherical-closure}. Notice that $\X$ and ${\X'}$ share the same valuation cone and the same set of colors via pullback. It suffices to check statement on ${\X'}$, which is given by \cite[Lemma 2.1.2]{Brion07}.
\end{proof}

\begin{proposition}\label{prop:character-H}
When $\ch\kk=0$, we have the following short exact sequence:
\[0\to \Lambda\to \pic(\u{\X})\to \fX^*(H)\to 0.\]
\end{proposition}
\begin{proof} 
The case when $\u{\X}$ is smooth is proved in \cite[Proposition 2.2.1]{Brion07}. The proof for the general case is similar. Let $\psi:G\to G/H$ be the quotient map. The following short exact sequence in \cite[(2.2.5)]{Brion07} still holds in our case:
\[0\to \fX^*(B)^{B\cap H}\to \fX^*(B)\times_{\fX^*(B\cap H)}\fX^*(H)\to \fX^*( H)\to 0.\]
Here the first term $\fX^*(B)^{B\cap H}$ is the set of weights of $B$-eigenfunctions on $\u{\X}$, hence is $\Lambda$. The second term $\fX^*(B)\times_{\fX^*(B\cap H)}\fX^*(H)$ is the set of weights of $B\times H$-eigenfunctions on $G$.  It is freely generated by the $B\times H$-weight of $f_{\uD}$, $\uD\in \Delta$, where $f_{\uD}$ is the defining equation of $\psi^{-1}(\uD)$ on $G$ with $f_{\uD}(o) = 1$, see \cite[Lemma 6.2.2]{Luna}. Furthermore, the following diagram commutes because $\nu_{\uD}(f)=\nu_{\psi^{-1}(\uD)}{(\pi^*f)}$.
$$\begin{CD}\fX^*(B)^{B\cap H}@>>> \fX^*(B)\times_{\fX^*(B\cap H)}\fX^*(H)\\
@| @|\\
\Lambda@>>>\pic(\u{\X})\end{CD}$$
 The proposition then follows.
\end{proof}

To summarize, we have 
\begin{thm}\label{thm:char0}
Assume that $\ch\kk=0$, and \R{that} $\X$ is the log smooth variety associated to a wonderful compactification of an open sober spherical homogeneous space $G/H$ with $G$ simply connected. Further assume all colors of $\u{\X}$ are of type (b). Then we have:
\begin{enumerate1}
\item Any log admissible effective curve class on $\u{\X}$ \R{can be represented by a free log rational curve on $\X$.}
\item The unique closed orbit is a fully free center of $\X$. In particular, $\X$ is $\A^1$-connected.
\item The center of $\X$ is primitive if and only if $\fX^*(H)$ is torsion-free. In particular, when $H$ is connected, the center of $\X$ is primitive.
\end{enumerate1}
\end{thm}
\begin{proof}
The statements follow from Theorem \ref{thm:main-wonderful} and Proposition \ref{prop:one-cycle}, \ref{prop:cone} and \ref{prop:character-H}. Note that $H$ is connected implies $\fX^*(H)$ is torsion-free.
\end{proof}






\subsection{Semisimple groups}\label{ss:ssgroup}
Let $G$ be a semisimple linear algebraic $\kk$-group of rank $r$, and $B$ \R{be} a Borel subgroup containing a maximal torus $T$. Denote by $E := \fX^{*}(T)\otimes_{\ZZ}\QQ$ the $\QQ$-vector space with the Euclidean product $(\cdot,\cdot)$ given by the Killing form. Let $(E, \Phi)$ be the root system associated to $(G,T)$ with the set of positive simple roots $\Delta:=\{\alpha_1,\cdots,\alpha_r\} \subset \Phi$ associated to $B$. Let $\Lambda_R$ be the root lattice generated by $\Delta$. Let $\Lambda_{R^\vee}$ be the coroot lattice in $E$ generated by the coroots 
\[
\alpha_j^\vee=\frac{2}{(\alpha_j,\alpha_j)} \alpha_j, \text{ }j=1,\cdots, r.
\]


Let $\mathcal{C}_+$ be the positive Weyl chamber spanned by the positive linear combination of fundamental weights. Denote by $\cC_{-}:= - \cC_{+}$ the negative Weyl chamber. 


The group $G$ is a separable spherical homogeneous space under $G\times G$ by $(g,h).k=gkh^{-1}$, because it contains a $B\times B^-$-dense orbit by Bruhat decomposition. Denote $\u{\X}_G$ by the wonderful compactification of $G$ and denote $\X_G$ by the associated log variety. 

\begin{proposition}\label{prop:lattice}
We have the following properties:
\begin{enumerate}
\item $\Lambda\cong \fX^*(T)$ and $N\cong \fX_*(T)\otimes_\Z \Q$;
\item $\cV$ is the negative Weyl chamber;
\item the set of colors $\Delta$ maps bijectively to the set of simple coroots under $\rho|_\Delta$ and the image of $\rho|_\Delta$ is the coroot lattice in $N$.
\end{enumerate}
\end{proposition}

\proof In characteristic zero, the first two statements were proved by Brion \cite[Section 3.1, 4.1]{Brion-spherical}. Since Brion's argument only uses basic properties of spherical embeddings \cite{Knop} and the local structure theorem \ref{thm:local-s}, \R{(1) and (2) hold in arbitrary characteristic}.

For (3), it suffices to \R{consider} the case where $G$ is simply connected. In this case, we know that each defining equation $f_i$ of $\uD_i$ is a $B\times B^-$-eigenfunction with the weight $(\chi_i,-\chi_i)$, where $\chi_i$ is a fundamental weight \cite[Proposition 6.1.11]{Brion-fbook}. Since $\Lambda$ is generated by the fundamental weights and $\nu_{\uD_i}(f_j)=\delta_{ij}$, we conclude that $\rho(\uD_i)$ is a simple coroot.\qed

\begin{thm}\label{thm:G}
Let $G$ be a semisimple linear algebraic $\kk$-group with arbitrary $\ch \kk \geq 0$, and $\X_G$ be the log smooth variety associated to the wonderful compactification of $G$. Let $\Lambda_{R^\vee}$ be the coroot lattice of $G$ and $\cC$ be a positive Weyl chamber of the root system of $(G,T)$ with a maximal torus $T$. Then we have the following:
\begin{enumerate1}
\item Any log-admissible effective curve class on $\u{\X}_G$ can be represented by a free $\A^1$-curve on $G$. Furthermore, we have 
$\NE(\X)=\Lambda_{R^\vee}\cap \cC^+.$
\item  If \R{we} further assume that $\ch\kk\nmid |\pi_1(G,T)|$, then the unique closed orbit is a fully free center of $\X_G$. In particular, $\X_G$ is separably $\A^1$-connected.
\item  The unique closed orbit as the center of $\X$ is \R{primitively free} if and only if $G$ is simply connected. In this case, it is separably $\A^1$-connected in arbitrary characteristic.
\end{enumerate1}
\end{thm}
\begin{proof}
We check \R{that} all hypotheses \R{of} Theorem \ref{thm:main-wonderful} are satisfied. Clearly the center $Y \cong G\times G/B\times B^- \subset \uX_G$ is $G$-separable. Hypothesis \ref{hyp:cone} follows from Proposition \ref{prop:lattice}. Hypothesis \ref{hyp:knop} can be verified by an argument similar to the proof of \cite[Theorem 2.4]{irr-rat}. We may consider the $G$-equivariant morphism $\u{\X}_G \to \u{\X}_{G_{ad}}$ where the latter is {smooth} \cite[Theorem 6.1.8]{Brion-fbook}. By Bruhat decomposition \cite[Theorem 22.6]{Borel}, the pullback of a color on $\u{\X}_{G_{ad}}$ is a (reduced) color on $\u{\X}_G$. Applying the projection formula as in the proof of \cite[Theorem 2.4]{irr-rat}, we get a basis of $B$-invariant curves for $NE(\u\X)$. All such curves lie on the center $\uY$ because the Weyl group acts transitively on the torus fixed points $\u\X^{T}$. See \cite[Theorem 2]{Brion-onecycle} and \cite[Proposition 6.2.3]{Brion-fbook}. Now the theorem follows from Lemma \ref{lem:back-to-log-class}, Proposition \ref{prop:lattice}, and Theorem \ref{thm:main-wonderful}. 
\end{proof}



\section{Logarithmic Hartshorne's conjecture}\label{sec:positive-tangent}

In this section, we fix a connected log smooth (possibly with singular underlying structure) variety $X$. Let $r: Y \to X$ be a birational log \'etale morphism of connected log smooth varieties with $\uY$ smooth. In particular, $Y$ has  normal crossings boundary. Denote by $\Delta_X$ and $\Delta_Y$ the boundary of $X$ and $Y$ respectively. Write $U = X \setminus \Delta_X = Y \setminus \Delta_Y$.

\begin{definition}\label{def:interior-ample}\Yi{(1.63) fixed by ample}
A line bundle $L$ over $X$ is called \R{\emph{interior-positive}} if $L \cdot \uC > 0$ for every irreducible proper curve $\uC$ on $\uX$ such that $U \cap \uC \neq \emptyset$.
\end{definition}

\begin{lemma}\label{lem:log-canonical-inv}
If $-K_X$ is interior-positive over $X$, then $-K_{X'}$ is interior-positive over $X'$ for any log \'etale birational morphism $\phi: X' \to X$.
\end{lemma}
\begin{proof}
This is follows from the projection formula and $\phi^*K_X=K_{X'}$.
\end{proof} 

\begin{lemma}\label{lem:rational-curve}
Assume that  $-K_{X}$ is \R{interior-positive}. For any point $x \in U$, there is an irreducible rational curve $\uZ \to \uX$ such that $x \in \uZ$, and $-K_{X} \cdot \uZ \leq \dim X + 1$, where the equality holds only if $\uZ\to \uX$ factors through $U$. 
\end{lemma}
\begin{proof}
We first choose an arbitrary usual stable map $\uf: \uC \to \uX$ through $x$ with $\uC$ smooth and irreducible, hence a stable log map $f: C \to X$. Taking proper transform, we obtain a stable log map $f' : C \to Y$ with $f = f' \circ r$. Since $r$ is log \'etale, we obtain:
\[
(-K_{\uY})\cdot C = (-K_{Y})\cdot C + \Delta_Y \cdot C \geq -K_{Y} \cdot C = -K_{X} \cdot C > 0.
\]
By Mori's Bend-and-Break, there is a rational curve $\uZ \to \uY$ through $x$ with $(-K_{\uY})\cdot \uZ \leq \dim Y + 1$. Composing with $r$, we obtain a rational curve $\uZ \to \uX$ through $x$. Then we calculate
\[
(-K_{X})\cdot \uZ = (-K_{Y}) \cdot \uZ = (-K_{\uY}) \cdot \uZ - \Delta_Y \cdot \uZ \leq \dim X + 1.
\]
Here the equality holds only if $\Delta_Y \cdot \uZ = 0$. This concludes the proof.
\end{proof}

The next result is due to Keel-McKernan \cite{KM}. Here we slightly strengthen the result to fits our needs.

\begin{proposition}\label{prop:bend-break}
Assume that $-K_{X}$ is \R{interior-positive}. Then for any closed point $x \in U$ there is a genus zero stable log map $f: C \to X$ through $x$ such that it is minimal with respect to a fixed polarization $H$ among all \R{curves that contain $x$} and $f$ \R{satisfies} one of the following properties:\Qile{Can one rephrase this sentence? not sure what u mean, rephrased a bit}
\begin{enumerate}
 \item $f$ is a $\PP^1$-curve with $\deg f^*(-K_{X}) \leq n + 1$;
 \item $f$ is an $\AA^1$-curve with $\deg f^*(-K_{X}) \leq n$.
\end{enumerate}
\end{proposition}
\begin{proof}
By Lemma \ref{lem:rational-curve}, we choose an irreducible rational curve in $\uX$ through $x$. We then have a genus zero log map $f: C \to Y$. 
We may assume that $\deg f^*H$ is minimal.
By the usual Bend-and-Break lemma, we may further assume that $(-K_{\uY}) \cdot C \leq n+1$. The {\em size} at a contact marking is the length of $f^*\Delta_{Y}$. We may assume that $f$ achieves the maximal size at a contact marking among all such genus zero stable log maps through $x$.

Assume that $f$ is not an $\AA^1$ or $\PP^1$-curve. Let $k \geq 2$ be the number of contact markings of $f$. We also mark one point $t$ mapped to $x$. Let $\fM$ be the moduli stack of stable log maps with discrete data giving by $f$. Since 
\[
(-K_{Y}) \cdot C = (-K_{X}) \cdot C > 0,
\]
we calculate
\[
\dim_{[f]}\fM(t) \geq (-K_{Y}) \cdot C + k+1 - 3 >  k -2 \geq 0.
\]
where $\fM(t) \subset \fM$ denotes the closed substack fixing the marking $t$. Thus the image of $f$ moves fixing $t \mapsto x$. 

Since $f$ is minimal with respect to a polarization $H$, we obtain a ruled surface $\pi: S \to B$ with $B$ a smooth proper curve, a morphism $h: S \to \uY$ obtained by deforming $f$, and at least three sections of $\pi$ with one $\Sigma_1$ given by the marking $t$, and at least two other contact markings, write $\Sigma_2$ and $\Sigma_3$ with one given by the contact marking of maximal size. Since $\Sigma_1$ is contractible, it has negative self-intersection. Since $S$ is a ruled surface, $\Sigma_2$ and $\Sigma_{3}$ have positive self-intersections. But the choice of maximal size implies that $\Sigma_2 \cdot \Sigma_3 = 0$. This is a contradiction, since the effective cone of $S$ is two dimensional.

Finally the statement on the degree of $f^*(-K_{X})$ can be deduced from the following calculation:
\[
\deg f^*(-K_{X}) = \deg f^*(-K_{Y}) = (-K_{\uY}) \cdot C - \Delta_Y \cdot C \leq n+1 - \Delta_Y \cdot C.
\]
\end{proof}

\begin{thm}\label{thm:projective-space}
Let $X=(\uX,\Delta_X)$ be a log smooth projective variety such that
\begin{enumerate}
 \item $-K_{X}$ is interior-positive;
 \item there exists a point $x \in U$ such that for any $\AA^1$ or $\PP^1$-curve $f: C \to X$ through $x$, the pull-back $f^*T_{X}$ is ample.
\end{enumerate}
Then we have:
\begin{enumeratea}
 \item if there is an $\AA^1$-curve through $x$, then $\uX =\PP^n$ and $\Delta_X$ is a hyperplane;
 \item otherwise, there is a $\PP^1$-curve through $x$, and $X = \PP^n$ with $\Delta_X = \emptyset$.
\end{enumeratea}
\end{thm}

We first observe the following:

\begin{lemma}\label{lem:resolution-ample}
Notations and assumptions as in Theorem \ref{thm:projective-space}. Let $f: C \to Y$ be an $\AA^1$ or $\PP^1$-curve through $x$ as in Proposition \ref{prop:bend-break}. Then 
$f$ is an immeerrsion and  
\[
f^*T_{\uY} = \cO(2) \oplus \cO(1) \oplus \cdots \oplus \cO(1).
\]
\end{lemma}
\begin{proof}
The natural inclusion $T_{Y} \to T_{\uY}$ induces the following exact sequence:
\[
0 \to f^*T_{Y} \to f^*T_{\uY} \to \cT \to 0
\]
where $\cT$ is a torsion sheaf supported on the contact marking. Since $f^*T_{Y}$ is ample, the pull-back $f^*T_{\uY}$ is ample as well. By Proposition \ref{prop:bend-break}, we have $\deg f^*T_{\uY} \leq n+1$. Since $f^*T_{\uY}$ is ample, the splitting type of $f^*T_{\uY}$ then follows. The immersion follows from \cite[Chapter IV 2.11]{kollar}.
\end{proof}

\begin{proof}[Proof of Theorem \ref{thm:projective-space}]\Yi{(1.67) fixed }
Take the stable log map $f: C \to Y$ as in Lemma \ref{lem:resolution-ample}. 
If $f$ is a very free minimal $\PP^1$-curve through $x$, then any deformation of $f$ can not break or completely fallen into the boundary,  as it is minimal, and through a point in $U$. Thus any deformation of $f$ does not meet $\Delta_Y$. Using Lemma \ref{lem:resolution-ample} and the identical argument as in \cite[Chapter V 3.7]{kollar}, we deduce that $\uY \cong \PP^n$.  Then $\Delta_Y$ hence $\Delta_X$ is empty as any effective divisor in $\PP^n$ is ample. This implies that there is an \'etale birational map $\PP^n \to X$, from which we deduce that $X = \PP^n$.

Assume $f: C \to Y$ is an $\AA^1$-curve, and is minimal as in Lemma \ref{lem:resolution-ample}. The ampleness of $f^*T_{Y} \cong (r\circ f)^*T_{X}$ implies that $\deg f^*(-K_{Y}) \geq n$, hence $\deg f^*(-K_{Y}) = n$ by Proposition \ref{prop:bend-break}. We then calculate
\[
\Delta_Y \cdot [C] = (-K_{\uY}) \cdot C -  (-K_{Y}) \cdot C = n + 1 - n = 1.
\]
This implies that $f$ intersects $\Delta_Y$ transversally at the unique contact marking. We also observe that any deformation of $f$ can not break or completely lie on the boundary. Thus, deforming $f$ is the same as deforming the underlying stable maps, i.e. they are parameterized by the same moduli space. The same argument as in \cite[Chapter V 3.7]{kollar} implies $\uY \cong \PP^n$. The degree consideration implies that $\Delta_Y$ is a hyperplane in $\PP^n$. Thus $r$ has to be an isomorphism. 
\end{proof}


%
%


\providecommand{\bysame}{\leavevmode\hbox to3em{\hrulefill}\thinspace}
\providecommand{\MR}{\relax\ifhmode\unskip\space\fi MR }
\providecommand{\MRhref}[2]{%
  \href{http://www.ams.org/mathscinet-getitem?mr=#1}{#2}
}
\providecommand{\href}[2]{#2}

\end{document}